\documentclass{article}
\usepackage[utf8]{inputenc}
\usepackage{amsmath}
\usepackage{indentfirst}
\usepackage{amsthm}
\usepackage{tikz}
\usepackage{changepage}
\usepackage{url}
\usepackage{titling}
\usepackage{ragged2e}

\usepackage[maxbibnames=9,sortcites, doi=false, url=false, backend=biber]{biblatex}
 \usepackage{stackengine,amssymb,graphicx,scalerel}
\theoremstyle{definition}
\newtheorem{definition}{Definition}[section]
\numberwithin{definition}{section}
\newtheorem{theorem}{Theorem}[section]
\numberwithin{theorem}{section}
\newtheorem{corollary}{Corollary}[theorem]
\newtheorem{lemma}{Lemma}[section]
\numberwithin{equation}{section}
\usetikzlibrary{backgrounds}
\theoremstyle{remark}
\newtheorem*{remark}{{Remark}}
  
\title{\bf Evolution of mesoscopic interactions and scattering solutions of the Boltzmann equation}
\author{Nima Moini \\ University of California, Berkeley}
\date{}
\begin{document}

\begin{titlingpage}
\maketitle
\begin{abstract}
The purpose of this paper is to study the evolution of moving interacting particles on the mesoscopic scale. We will introduce an uncertainty principle and a new priori bound for the evolution of particles subject to a general mesoscopic interaction and use this setting to  demonstrate evidence for the universal dispersion of particles in the absence of the conventional methods like velocity averaging lemmas or Strichartz-type estimates. We will develop a scattering theory in generality and use its frame work to show the existence and uniqueness of a class scattering solutions to the Boltzmann equation with asymptotic stability and completeness in the $L^{\infty}$ setting.
\end{abstract}
\setcounter{tocdepth}{1}
\tableofcontents
\end{titlingpage}

\section{Introduction}
The purpose of this paper is to study the evolution of interacting moving bodies on the mesoscopic scale. Consider a differential equation of the form:   
\begin{equation}
\begin{aligned}
    \partial_t f(x, \xi, t)+\xi.\nabla_x f(x, \xi, t) &= I(f, x, \xi, t) \\  f(x, \xi, 0)&=f_{0}(x,\xi)
\end{aligned}
\end{equation}
In the equation defined above, $x,\xi \in {\rm I\!R}^n$ and $t \in [0,\infty)$. We expect $f(x,\xi,t)$ to be non-negative and interpret it as the density or amplitude of bodies or particles with velocity $\xi$ located at $x$. The left hand side of this equation resembles the principle of inertia, which implies that particles will get transported linearly along the trajectory of their velocities, while the right hand side is representative of change via possibly non linear interactions of particles. The interaction $I(f,x,\xi,t)$ is expected be a mesoscopic interaction in the sense defined below.

\begin{definition}\label{interaction} $I(f,x,\xi,t)$ is a {\it mesoscopic interaction} if: 
\begin{align*}
                        \int_{{\rm I\!R}^n} I(f, x, \xi, t) \ d\xi&= 0 \\ \int_{{\rm I\!R}^n} I(f, x, \xi, t) |\xi|^2 \ d\xi&=0 \\ \int_{{\rm I\!R}^n} I(f, x, \xi, t) \xi_{i} \ d\xi&=0 \ \ \ 1\leq i \leq n
\end{align*}
\end{definition}
\indent The characteristic of the partial differential equation (1.1) initiating from arbitrary $x,\xi \in {\rm I\!R}^n$ is the line $(x + t\xi,\xi, t) \subset {\rm I\!R}^n\times{\rm I\!R}^n\times {\rm I\!R}$. It is possible to interpret this equation as an infinite dimensional system of ordinary differential equations which start from the same initial value $f_{0}(x,\xi)$ and evolve along these characteristics: 
\begin{equation}
\begin{aligned}
    \frac{d}{dt}f(x+t\xi,\xi,t)&=I(f,x+t\xi,\xi,t) \\
    f(x,\xi,0)&=f_{0}(x,\xi)
\end{aligned}
\end{equation}
\indent One immediate consequence of the equation above and the definition of a mesoscopic interaction is that, for the solutions of equation (1.1), the total amount of mass, momentum and energy remain invariant. For example, the conservation of mass can be proven by the argument below and the other two conservation laws are similar.
\begin{multline*}
    \frac{d}{dt}\int_{{\rm I\!R}^n}\int_{{\rm I\!R}^n} f(x, \xi, t)  \ dxd\xi= 
\int_{{\rm I\!R}^n}\int_{{\rm I\!R}^n}\frac{d}{dt}f(x+t\xi,\xi,t)\ dxd\xi=\\ \int_{{\rm I\!R}^n}\int_{{\rm I\!R}^n} I(f, x+t\xi, \xi, t)  \ dxd\xi=\int_{{\rm I\!R}^n}\int_{{\rm I\!R}^n} I(f, x, \xi, t)  \ dxd\xi=0 
\end{multline*}

\indent The narrative of this writing is grounded in concrete examples and the generality of the theory is built upon that. Some of the most important examples of mesoscopic interactions are the linear transport equation for $I=0$ and different variants of the Boltzmann equation where $I$ is the Boltzmann collision operator. The objective here is to study the mesoscopic scale in more generality and then return to the Boltzmann equation in the final section. It is not possible to provide a complete description of the new results in this introduction, because of their dependence on new definitions that will appear alongside the new methods. We will briefly compare the results of each section to the relevant mathematical literature in the field:  \\ 

\indent {\bf Section 2:} We will introduce  the  concepts  of  {\it uncertainty} and {\it angular momentum} associated to the solutions of equation (1.1) and define a new priori bound named {\it relative angular norm}. Under a set of minimal  a priori assumptions on mass, momentum and energy, the results of this section remain valid for an arbitrary mesoscopic interaction, including the Boltzmann equation. 
    \begin{itemize}
        \item Angular momentum increases linearly over time with a slope equal to the amount of total energy, furthermore the relative angular norm remains bounded and the uncertainty goes to infinity with time.
         \item The total amount of energy within any bounded set of the spatial variable is integrable over time, and the evolution of moving interacting particles is subject to dispersion. In particular, particles with different velocities have a tendency to separate from one another.
        \item The total amount of mass, as time goes to infinity, will concentrate inside a specific collection of cones with an arbitrarily small apex angle. This concentration and the increasing of uncertainty are tailored together. 
        
        \item[$\diamond$] The mathematical constructs of Section 2 develop an area between kinetic theory and quantum mechanics which was previously a no man's land. We will not use the conventional methods of kinetic theory such as velocity averaging, Strichartz estimates or the notion of entropy \cite{MR2083636,MR1942465,MR2043752}. The general formalism of equation (1.1) has been studied before by Lions, Perthame, Tadmor, Tao, Tartar and other authors using these  methods \cite{MR1101105,MR1191915, MR1266203,MR1099693,MR1201239,MR1284790,tartar,MR1166050,tadmortao,MR923047,MR1127927}. They develop a kinetic point of view for the single or multi-dimensional systems of conservation laws, comparable to equation (1.1) subject to a mesoscopic interaction. These authors achieve integrability results and demonstrate dispersive and regularizing effects. Their ideas inspired other mathematicians and appear repeatedly in the literature of the field \cite{imbert2020regularity, arsenio2020energy, MR3551261, MR2083859}. In contrast, we will demonstrate dispersion based on an entirely different intuition, namely a principle of uncertainty, the angular momentum and Morawetz type estimates. The methods of this paper have analogies to the theory of linear and non linear Schrodinger’s equations, as Tao \cite{MR2233925} describes very well. We will achieve integrability results and demonstrate dispersion in the absence of these conventional methods.
\end{itemize}

\indent {\bf Section 3:} We will investigate the role of a norm in measuring dispersion of particles and develop a scattering theory for a general mesoscopic interaction. In order to do so we will provide two new definitions, {\it observer's norm} and {\it scattering norm}. In addition, we will introduce two concepts of solution for equation (1.1), {\it scattering solution} and {\it uniform scattering solution}. 
    \begin{itemize}
    \item Any scattering solution will converge to a unique solution of the linear transport equation in a specific sense. This convergence will be uniform over the spatial variable for a uniform scattering solution.  
    \item We will reveal how the contents of Section 3 provide a natural route for creating scattering solutions. These ideas will be used in Section 4 in the specific case of the Boltzmann equation. 
    \item[$\diamond$] Bardos, Gamba, Golse and Levermore \cite{MR3535892} have obtained results on scattering theory of the Boltzmann equation using $L^1$ norms. In particular, they prove an almost everywhere convergence to a unique linear state and demonstrate that this linear state is not necessarily the state of minimal entropy. We will follow a distinct approach that when paired with the findings of Section 4, will improve results in the $L^{\infty}$ setting. For instance we will show existence of solutions that scatter to linear states arbitrarily close to any given solution of the linear transport equation chosen from an appropriate space.
    \end{itemize}
    
    \indent {\bf Section 4:} In the last section we will use the insights from Section 3 to introduce two new concepts, {\it scattering frame} and {\it uniform scattering frame}. We will define an {\it isotropic radius} associated to these frames and use this setup to create a class of  global in time solutions to the Boltzmann equation, namely  the scattering solutions. The results of Section 2 and 3 are applicable to this class of solutions.
          
\begin{itemize}
        \item Accompanying any scattering frame (uniform scattering frame) there exists a unique scattering (uniform scattering) solution associated to any initial value within the isotropic radius of that frame.
        
        \item We will show that the Maxwellian can be used as a scattering frame for the Boltzmann equation and conclude existence and uniqueness of scattering solutions for the Boltzmann equation in the $L^{\infty}$ setting.
        
         \item[$\diamond$] Illner, Shinbrot, Arsenio and others \cite{MR760333,MR2770020,MR3535892} create global in time solutions to the Boltzmann equation for sufficiently small initial values based on the idea of comparison with a linear state. In this paper an initial value is measured against a scattering frame of reference and the concept of sufficiently small is quantified by the notion of isotropic radius. This setup makes the findings of Section 3 applicable in Section 4. We will improve a classical calculation (Lemma \ref{lem}) linked to the existence theory of the Boltzmann equation and create a class of global in time solutions which have scattering theory in the $L^{\infty}$ setting with asymptotic completeness and stability results.    
         \end{itemize}

\indent  The rest of this section will provide motivation for analyzing equation (1.1) when it is subject to a general mesoscopic interaction. We will describe the importance of the Boltzmann equation in mathematics while providing evidence that its signature properties are identical for a general mesoscopic interaction. In particular, this similarity is fundamentally rooted in conservation laws and therefore not reliant on the specific structure of the Boltzmann equation. One could argue that the role of the Boltzmann equation in kinetic theory is similar to that of the Bernoulli trials within the theory of probability. \\ 

\indent When we think of particles as discrete elastic bodies, that is the microscopic scale. In that setting, classical mechanics governs the evolution of interacting particles, namely the principle of inertia and conservation laws of momentum and energy. On the other hand, we have partial differential equations that represent the evolution of macroscopic descriptions, where particles are no longer discrete objects. The crucial examples of such equations are Navier–Stokes and Euler equations in the context of fluid dynamics. The role of the mesoscopic scale, with the Boltzmann equation as its central example, is to create a connection between the microscopic and macroscopic descriptions. \\

\indent A rigorous  derivation of the Boltzmann equation from the dynamics of the microscopic scale was created first by Lanford \cite{MR0441164, MR0479206} in an asymptotic regime that is often called the Boltzmann-Grad limit \cite{MR3608289, MR4131018, MR3625187}, which is subject to a short time interval of validity. Without limitation on time but with extra constraints on the initial data, Illner and Pulvirenti prove the validity of the Boltzmann equation for a rare gas in a vacuum in dimensions 2 and 3  based on the notion of single particle distribution and BBGKY-hierarchy \cite{MR913933, MR985619}. Lanford's ideas were later improved in the landmark works of Raymond, Gallagher and others \cite{gallagher2013newton}. Another successful and unique approach for the justification of the Boltzmann equation, without limitation on time but under stochastic assumptions for the microscopic dynamics, appears in the works of Rezakhanlou \cite{MR2076921}. Lanford \cite{MR0441164} writes: \\

\begin{adjustwidth}{25pt}{25pt}
{\it \small
The conceptual foundations of the Boltzmann equation seem
to me to merit careful study not so much for their own sake as because
the Boltzmann equation is a prototype of a mathematical construct central to the theory of time dependent phenomena in large systems}\\
\end{adjustwidth}

\indent Similar to the attempts for justification of the Boltzmann equation from the microscopic scale, there has been a vast amount of research dedicated to the relationship between the Boltzmann equation and the aforementioned higher scale macroscopic equations. Assuming the validity of laws of classical mechanics in the microscopic scale is self-evident, the hierarchy described before demonstrates that the Boltzmann equation can be a scaffold, to illuminate the validity of the most important equations appearing in fluid mechanics \cite{MR1842343}, which were originally derived independent of the Boltzmann equation. The transition from a mesoscopic scale to a macroscopic one is called a hydrodynamic limit \cite{MR2683475, MR3150645}. More than often, mathematical arguments related to this transition remain valid if we replace the Boltzmann collision operator with an arbitrary mesoscopic interaction. For example in \cite{senji, MR1842343}, we see that it is possible, at least formally, to derive the compressible Euler equations from the Boltzmann equation, using some constitutive relations in fluid dynamics, yet an identical argument is also valid for any solution of equation (1.1) with a mesoscopic interaction. \\

Intrinsic properties of the Boltzmann equation, independent of their relationship to the microscopic or macroscopic scales, lead to a series of heuristic arguments that establish the relevance of the equation to the actual physical problem of interacting bodies. For example, conservation of mass, momentum and energy remain true for any mesoscopic interaction and are not solely a property of the Boltzmann collision operator. Another feature of the Boltzmann equation is the monotonicity of the famous Boltzmann's entropy formula, which can be interpreted as a mathematical formulation of the second law of thermodynamics. However we will show that in general the solutions of the Boltzmann equation do not converge to a state of thermodynamic equilibrium or Maxwellian. It is interesting to point out that the concept of Maxwellian as an equilibrium state existed before Boltzmann. It is possible to argue that the specific structure of the Boltzmann collision operator which leads to the monotonicity of the Boltzmann's entropy formula is a result of a genius reverse engineering. The conservation laws are primordial compared to the specific structure of the interactions. We will show there exists other monotone quantities besides entropy for any mesoscopic interaction, that one can harness to study the evolution of solutions to equation (1.1) as well as the Boltzmann equation. Tartar \cite{MR2397052} writes: \\

\begin{adjustwidth}{25pt}{25pt}
{\it \small
Of course, many consider the Boltzmann equation as an answer[...]but the Boltzmann equation has a similar defect than thermodynamics[...]. In order to study how irreversibility occurs it is important to start from conservative models which are reversible[...]and to study if something has really been lost, or if a better account can be given of what other models declare lost.}\\
\end{adjustwidth}

\indent We intend this writing to be clear and accessible for a reader with background knowledge in the field and hope that it reflects the sage advice of Halmos \cite{MR277319} on mathematical writing. \\ 
\subsection*{Acknowledgements}
I want to express my gratitude for the outstanding mentorship I received throughout the creation of this work from Daniel Tataru. In addition, I am deeply thankful for the enduring encouragement of Mehrdad Shashahani as well as the love and support of Amanda Jeanne Tose (AJT), Kaveh and Zahra.

\newpage
\section{Dispersive evolution of interacting particles}
Throughout this section we assume $f$ is a non negative solution of equation (1.1) for some mesoscopic interaction $I$. We expect that at time zero the mass, momentum and energy as represented below are well defined quantities. 

\begin{equation}
\begin{aligned}
  ({\it mass})&  &{\bf M} &=\int_{{\rm I\!R}^n}\int_{{\rm I\!R}^n} f(x,\xi, 0) \ dxd\xi < \ \infty  \\
     ({\it energy})&  & {\bf E} &=\int_{{\rm I\!R}^n}\int_{{\rm I\!R}^n} f(x,\xi, 0)|\xi|^2 \ dxd\xi < \ \infty
     \\
     ({\it momentum})&  &{\bf V} &=\int_{{\rm I\!R}^n}\int_{{\rm I\!R}^n} f(x,\xi, 0)\xi \ dxd\xi < \ \infty
\end{aligned}
\end{equation}  

Since $|\bf{V}| \leq \frac{1}{2}(\bf{M} + \bf{E})$, the boundedness of mass and energy implies the boundedness of momentum. As a consequence of Definition \ref{interaction} these quantities remain invariant over time. Furthermore we assume that at time zero the following integral, which  represents the localization of mass, is convergent:
\begin{align}
     \int_{{\rm I\!R}^n}\int_{{\rm I\!R}^n} f(x,\xi, 0)|x|^2 \ dxd\xi < \infty  
\end{align}

\indent Results that appear in this section hold true under the assumptions above and are independent of the specific structure of the mesoscopic interaction $I$, consequently they are true for the solutions of the Boltzmann equation as well. The discussion about the existence of such solutions is omitted until the next sections.

\begin{definition}\label{angular momentum}
 Let $A(t)$ be {\it the angular momentum} of $f$ at time $t$: 
\begin{align*}
    A(t) = \int_{{\rm I\!R}^n}\int_{{\rm I\!R}^n} f(x,\xi, t)x\cdot\xi \ dxd\xi  
\end{align*}
\end{definition}

\begin{theorem}\label{angularthm} For any solution of equation (1.1) like $f$ satisfying (2.1), the angular momentum is a linear function of time increasing proportional to the total amount of energy({\bf E}):
\begin{align*}
    A(t)=\int_{{\rm I\!R}^n}\int_{{\rm I\!R}^n} f(x, \xi, t)x\cdot\xi \ dxd\xi = A(0) + t{\bf E}
\end{align*} 
\end{theorem}
\begin{proof} It is possible to compute the values of $f$ at an arbitrary time like $t$ by integrating the effects of interactions along the characteristics of the equation. We will carry out this idea using the change of variable $x \Longrightarrow x + t\xi$, it follows: 
\begin{multline*}
      A(t)=\int_{{\rm I\!R}^n}\int_{{\rm I\!R}^n} f(x, \xi, t)x\cdot\xi \ dxd\xi =  \int_{{\rm I\!R}^n}\int_{{\rm I\!R}^n} f(x + t\xi, \xi, t)(x+t\xi)\cdot\xi \ dxd\xi \\
    =\int_{{\rm I\!R}^n}\int_{{\rm I\!R}^n} f(x+t\xi, \xi, t)x\cdot\xi \ dxd\xi + t \int_{{\rm I\!R}^n}\int_{{\rm I\!R}^n} f(x + t\xi, \xi, t)|\xi|^2 \ dxd\xi
\end{multline*}
Consequently from equation (1.1) we get:  
\begin{multline*}
    A(t)= \int_{{\rm I\!R}^n}\int_{{\rm I\!R}^n} f(x, \xi, 0)x\cdot\xi \ dxd\xi + \int_{0}^{t} \int_{{\rm I\!R}^n}\int_{{\rm I\!R}^n} I(x+s\xi, \xi, s)x\cdot\xi \ dxd\xi ds \\ +t \int_{{\rm I\!R}^n}\int_{{\rm I\!R}^n} f(x, \xi, t)|\xi|^2 \ dxd\xi
\end{multline*}
\noindent Note that Definition \ref{interaction} implies: 
\begin{align*}
     \int_{{\rm I\!R}^n}\int_{{\rm I\!R}^n} I(x+s\xi, \xi, s)x\cdot\xi \ dxd\xi  = \int_{{\rm I\!R}^n}\int_{{\rm I\!R}^n} I(x, \xi, s)(x-s\xi)\cdot\xi \ dxd\xi = 0
\end{align*}
Using the computation above for $A(t)$ we get:
\begin{multline*}
         A(t)= \int_{{\rm I\!R}^n}\int_{{\rm I\!R}^n} f(x, \xi, 0)x\cdot\xi \ dxd\xi + \int_{0}^{t} \int_{{\rm I\!R}^n}\int_{{\rm I\!R}^n} I(x, \xi, s)(x-s\xi)\cdot\xi \ dxd\xi ds \\ + t \int_{{\rm I\!R}^n}\int_{{\rm I\!R}^n} f(x, \xi, 0)|\xi|^2 \ dxd\xi = A(0) + t{\bf E}
\end{multline*} 
\end{proof}

We will continue with a lemma which describes how localization of mass changes along the evolution of solutions to equation (1.1) and show that it is parabolic with respect to time.  

\begin{lemma}\label{consvlemma} For any solution of equation (1.1) like $f$ satisfying (2.1) and (2.2), the identities below hold true:
\begin{align*}
 {\bf i.}& \int_{{\rm I\!R}^n}\int_{{\rm I\!R}^n} f(x,\xi,t)|x|^2 \ dxd\xi=\int_{{\rm I\!R}^n}\int_{{\rm I\!R}^n} f(x,\xi,0)|x|^2 \ dxd\xi + 2tA(0) + t^2{\bf E} \\ 
     {\bf ii.}& \int_{{\rm I\!R}^n}\int_{{\rm I\!R}^n}  f(x, \xi, t)|x- t\xi|^2\ dxd\xi=\int_{{\rm I\!R}^n}\int_{{\rm I\!R}^n} f(x, \xi, 0)|x|^2\ dxd\xi
\end{align*}
\end{lemma}
\begin{proof} 
We will prove the second identity first. Start by integrating the effects of interactions along the characteristics of equation (1.1):
\begin{multline*}
            \int_{{\rm I\!R}^n}\int_{{\rm I\!R}^n}  f(x+t\xi, \xi, t)|x|^2\ dxd\xi  = \int_{{\rm I\!R}^n}\int_{{\rm I\!R}^n} f(x, \xi, 0)|x|^2\ dxd\xi\\ + \int_{0}^{t} \int_{{\rm I\!R}^n}\int_{{\rm I\!R}^n} I(f, x+s\xi, \xi, s)|x|^2\ dxd\xi ds
\end{multline*}
Therefore a change of variable implies:  
\begin{multline*}
      \int_{{\rm I\!R}^n}\int_{{\rm I\!R}^n}  f(x+t\xi, \xi, t)|x|^2 \ dxd\xi =\int_{{\rm I\!R}^n}\int_{{\rm I\!R}^n} f(x, \xi, 0)|x|^2\ dxd\xi \\+ \int_{0}^{t}\int_{{\rm I\!R}^n}\int_{{\rm I\!R}^n} I(f,x, \xi, s)|x - s\xi|^2 \ dxd\xi ds 
\end{multline*}
It follows from Definition \ref{interaction} that:
\begin{align*}
    \int_{{\rm I\!R}^n} I(f,x, \xi, t)|x - t\xi|^2 \ d\xi=\int_{{\rm I\!R}^n} I(f,x, \xi, t)(|x|^2 +t^2|\xi|^2 - 2t(x \cdot \xi)) \ d\xi=0
\end{align*}
We will substitute the equation above into the previous computation and deduce the second identity: 
\begin{align*}
\int_{{\rm I\!R}^n}\int_{{\rm I\!R}^n}  f(x+t\xi, \xi, t)|x|^2\ dxd\xi=\int_{{\rm I\!R}^n}\int_{{\rm I\!R}^n} f(x, \xi, 0)|x|^2\ dxd\xi
\end{align*}
To finish the proof, we will use the identity above and Theorem \ref{angularthm}:
\begin{multline*}
    \int_{{\rm I\!R}^n}\int_{{\rm
I\!R}^n} f(x,\xi,t)|x|^2 \ dxd\xi=\int_{{\rm I\!R}^n}\int_{{\rm
I\!R}^n} f(x+t\xi,\xi,t)|x+t\xi|^2 \ dxd\xi \\ = \int_{{\rm I\!R}^n}\int_{{\rm
I\!R}^n} f(x+t\xi,\xi,t)(|x|^2+t^2|\xi|^2 + 2t(x \cdot \xi)) \ dxd\xi \\ =  \int_{{\rm I\!R}^n}\int_{{\rm
I\!R}^n} f(x,\xi,0)|x|^2 dxd\xi + t^2{\bf E} +  2t\int_{{\rm I\!R}^n}\int_{{\rm
I\!R}^n} f(x,\xi,t)(x-t\xi) \cdot \xi \ dxd\xi \\=\int_{{\rm I\!R}^n}\int_{{\rm I\!R}^n} f(x,\xi,0)|x|^2 \ dxd\xi + 2tA(0) + t^2{\bf E}
\end{multline*}

\end{proof}

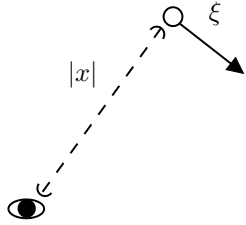
\begin{figure}[ht]
    \centering
    \label{fig:my_label}
\tikzset{every picture/.style={line width=0.75pt}} 

\begin{tikzpicture}[ x=0.75pt,y=0.75pt,yscale=-1,xscale=1]

\draw    (296.25,100.25) -- (327.14,124.4) ;
\draw [shift={(329.5,126.25)}, rotate = 218.02] [fill={rgb, 255:red, 0; green, 0; blue, 0 }  ][line width=0.08]  [draw opacity=0] (8.93,-4.29) -- (0,0) -- (8.93,4.29) -- cycle    ;
\draw  [fill={rgb, 255:red, 255; green, 255; blue, 255 }  ,fill opacity=1 ] (288.75,97.25) .. controls (288.75,94.63) and (290.88,92.5) .. (293.5,92.5) .. controls (296.12,92.5) and (298.25,94.63) .. (298.25,97.25) .. controls (298.25,99.87) and (296.12,102) .. (293.5,102) .. controls (290.88,102) and (288.75,99.87) .. (288.75,97.25) -- cycle ;
\draw   (210.5,194.25) .. controls (210.5,191.63) and (214.53,189.51) .. (219.5,189.51) .. controls (224.47,189.51) and (228.5,191.63) .. (228.5,194.25) .. controls (228.5,196.87) and (224.47,199) .. (219.5,199) .. controls (214.53,199) and (210.5,196.87) .. (210.5,194.25) -- cycle ;
\draw  [fill={rgb, 255:red, 0; green, 0; blue, 0 }  ,fill opacity=1 ] (215.59,194.07) .. controls (215.59,191.76) and (217.43,189.89) .. (219.7,189.89) .. controls (221.97,189.89) and (223.8,191.76) .. (223.8,194.07) .. controls (223.8,196.37) and (221.97,198.24) .. (219.7,198.24) .. controls (217.43,198.24) and (215.59,196.37) .. (215.59,194.07) -- cycle ;
\draw  [dash pattern={on 4.5pt off 4.5pt}]  (287.5,103) -- (226.7,187.07) ;
\draw [shift={(226.7,187.07)}, rotate = 305.88] [color={rgb, 255:red, 0; green, 0; blue, 0 }  ][line width=0.75]      (3.91,-3.91) .. controls (1.75,-3.91) and (0,-2.16) .. (0,0) .. controls (0,2.16) and (1.75,3.91) .. (3.91,3.91) ;
\draw [shift={(287.5,103)}, rotate = 125.88] [color={rgb, 255:red, 0; green, 0; blue, 0 }  ][line width=0.75]      (3.91,-3.91) .. controls (1.75,-3.91) and (0,-2.16) .. (0,0) .. controls (0,2.16) and (1.75,3.91) .. (3.91,3.91) ;

\draw (310,88) node [anchor=north west][inner sep=0.75pt]    {$\xi $};
\draw (239,118) node [anchor=north west][inner sep=0.75pt]    {$|x|$};

\end{tikzpicture}
\begin{adjustwidth}{63pt}{63pt}
\caption{\small Uncertainty associated to the measurement of momentum of a particle is proportional to its mass, velocity and distance to the observer.}
\end{adjustwidth}
\end{figure}

\begin{definition}
Let $U(t)$ be the {\it uncertainty} associated to $f$ at time $t$ relative to an idle observer at the origin:
\begin{align*}
    U(t) = \int_{{\rm I\!R}^n}\int_{{\rm I\!R}^n} f(x,\xi, t)|x||\xi| \ dxd\xi 
\end{align*}
\end{definition}

\begin{remark}It is possible to interpret the uncertainty defined above akin to its quantum mechanics counterpart, that is by highlighting a fundamental limit on the precision of the physical measurements. As in Figure 1, consider an observer located at the origin and assume the speed of light is $C$. For each particle located at $x$ there is a minimum delay of $T=C^{-1}|x|$ between the actual time of measurement and observation. The quantity $T\times f(x,\xi,t)|\xi|=C^{-1}f(x,\xi,t)|x||\xi|$ represents the uncertainty of measurement relative to an observer, due to this interval of delay for a particle at position $x$ with velocity $\xi$. After re-scaling of this quantity with $C=1$ and integrating over the space of the positions and velocities, one can get the definition above.   \\
\end{remark}

\begin{definition}
  Let $\|f\|_{G}$ be  the {\it relative angular norm} of $f$ defined in the terms of the uncertainty and angular momentum:
\begin{align*}.
    \|f\|_{G}=\sup_{t} \big(U(t) - A(t)\big)
\end{align*}
\end{definition}
\begin{theorem}\label{RANB} Under the assumptions (2.1) and (2.2), the relative angular norm of the solutions to equation (1.1) is bounded as below:
\begin{align*}
   \|f\|_G &\leq \int_{{\rm I\!R}^n}\int_{{\rm I\!R}^n} f(x, \xi, 0)(|x|^2+|\xi|^2)\ dxd\xi - \int_{{\rm I\!R}^n}\int_{{\rm I\!R}^n} f(x, \xi, 0)x\cdot\xi \ dxd\xi
\end{align*}
Furthermore, uncertainty goes to infinity with time: 
\begin{align*}
\lim_{t\rightarrow\infty} U(t) &= \infty
\end{align*}
\begin{proof}
Start by the computation below in which we use Theorem \ref{angular momentum}:
\begin{multline*}
      U(t) - A(t)=  
      \int_{{\rm I\!R}^n}\int_{{\rm I\!R}^n} f(x, \xi, t)|x||\xi| \ dxd\xi - \int_{{\rm I\!R}^n}\int_{{\rm I\!R}^n} f(x, \xi, t)x\cdot\xi \ dxd\xi \\ =\int_{{\rm I\!R}^n}\int_{{\rm I\!R}^n} f(x+t\xi, \xi, t)|x+t\xi||\xi| \ dxd\xi - A(0) - t\bf{E}
\end{multline*}
Using the triangle inequality and conservation of energy we get: 
\begin{multline*}
      U(t) - A(t) \leq \int_{{\rm I\!R}^n}\int_{{\rm I\!R}^n} f(x+t\xi, \xi, t)(|x| + t|\xi|^2) \ dxd\xi - A(0) - t\bf{E} \\ \leq \int_{{\rm I\!R}^n}\int_{{\rm I\!R}^n} f(x+t\xi, \xi, t)|x||\xi|\ dxd\xi - A(0)
\end{multline*}
We will apply Lemma \ref{consvlemma}, it follows: 
\begin{multline*}
      U(t) - A(t) \leq  \int_{{\rm I\!R}^n}\int_{{\rm I\!R}^n} f(x+t\xi, \xi, t)(|x|^2 + |\xi|^2)\ dxd\xi - A(0) \\ \leq \int_{{\rm I\!R}^n}\int_{{\rm I\!R}^n} f(x, \xi, 0)(|x|^2+|\xi|^2)\ dxd\xi - A(0)
\end{multline*}
The right hand side of the inequality above is independent of the time variable, it follows that: 
\begin{align*}
\|f\|_{G} &= \sup_{t}(U(t)-A(t)) \\ &\leq \int_{{\rm I\!R}^n}\int_{{\rm I\!R}^n} f(x, \xi, 0)(|x|^2+|\xi|^2)\ dxd\xi - \int_{{\rm I\!R}^n}\int_{{\rm I\!R}^n} f(x, \xi, 0)x\cdot\xi \ dxd\xi
\end{align*}
We just proved the desired bound for the relative angular norm. On the other hand Theorem \ref{angular momentum} implies  $\lim_{t\rightarrow\infty} A(t)=\infty$, therefore the uncertainty goes to infinity with time:  
\begin{align*}
    \lim_{t\rightarrow\infty} U(t)=\infty
\end{align*}
\end{proof}
\end{theorem}
\begin{corollary} The previous theorem implies that on average the uncertainty increases linearly over time, proportional to the total amount of the energy({\bf E}), although it may not be monotonic:
\begin{equation*}
    A(0) + t{\bf E} \leq U(t) \leq \int_{{\rm I\!R}^n}\int_{{\rm I\!R}^n} f(x, \xi, 0)(|x|^2+|\xi|^2)\ dxd\xi +  t{\bf E}
\end{equation*}
\end{corollary}
\indent We will use this uncertainty principle and the relative angular norm to demonstrate dispersive effects and prove results on the concentration of mass. The intuition behind the next definition is coming from the boundedness of relative angular norm and its utility will be shown in the forthcoming theorems.

\begin{definition}\label{bcone} Let  $C_{x_{0}}(x,c)$ be the {\it blind cone} at $x \in {\rm I\!R}^n$ with the apex angle $0<c$ relative to the observer $x_0 \in {\rm I\!R}^n$ as defined below. Here $\theta$ represents the angle between the two given vectors.   
\begin{align*}
    C_{x_0}(x,c) = \{\xi \in {\rm I\!R}^n \big | \theta(x-x_0, \xi) \notin  [c, \pi-c]\}
\end{align*}
 
\end{definition}
\begin{theorem}\label{dispersionE}
Assume $f$ is any solution of equation (1.1) such that satisfies (2.1). The energy contained inside any bounded subset of the spatial variable like $D \subset {\rm I\!R}^n$ is integrable over time:
\begin{align*}
    \int_{0}^{\infty}\int_{{\rm I\!R}^n}\int_{D} f(x,\xi, t)|\xi|^2 \ dxd\xi dt < \infty
\end{align*}
\end{theorem}
\begin{proof}
Let $\dfrac{(x- x_{0})\cdot\xi}{|x- x_0|}$ be the {\it localized angular momentum} at $x_0 \in {\rm I\!R}^n$. We will multiply both sides of equation (1.1) with this quantity: 
\begin{align*}
     \partial_t f(x,\xi,t)\frac{(x- x_{0})\cdot\xi}{|x- x_0|}+  \sum_{1}^{n}\frac{(x- x_{0})\cdot\xi}{|x-x_0|} \xi_{i}\partial_{x_{i}}f(x,\xi, t)=  I(f,x,\xi,t)\frac{(x- x_{0})\cdot\xi}{|x- x_0|}
\end{align*}
\noindent Integrate this equation over the space of positions and velocities and implement the change of variable $x \Longrightarrow x+x_{0}$, we get: 
     \begin{multline*}
     \int_{{\rm I\!R}^n}\int_{{\rm I\!R}^n} \partial_t f(x+x_0,\xi,t)\frac{(x\cdot\xi)}{|x|}\ dxd\xi\\ + \int_{{\rm I\!R}^n}\int_{{\rm I\!R}^n}  \sum_{1}^{n}\frac{(x\cdot\xi)}{|x|} \xi_{i}\partial_{x_{i}}f(x+x_0,\xi, t)\ dxd\xi \\= \int_{{\rm I\!R}^n}\int_{{\rm I\!R}^n} I(f, x+x_0,\xi,t)\frac{(x\cdot\xi)}{|x|}\ dxd\xi=0
     \end{multline*}
\noindent The last equality is true since $I$ is a mesoscopic interaction. Continue with an integration by parts with respect to $x$: 
     \begin{multline*} 
    \int_{{\rm I\!R}^n}\int_{{\rm I\!R}^n}\big( \frac{1}{|x|^3}(x\cdot\xi)^2 f(x+x_0,\xi, t)- \frac{1}{|x|}|\xi|^2 f(x+x_0,\xi, t) \big)\ dxd\xi\\+  \partial_t \int_{{\rm I\!R}^n}\int_{{\rm I\!R}^n}   f(x+x_0,\xi,t)\frac{(x\cdot\xi)}{|x|}\ dxd\xi=0
\end{multline*}
Let $\theta(x,\xi)$ be the angle between $x$ and $\xi$, we get:  
     \begin{multline*}
     \partial_t \int_{{\rm I\!R}^n}\int_{{\rm I\!R}^n} f(x+x_0,\xi,t)\frac{(x\cdot\xi)}{|x|}\ dxd\xi \\=
    \int_{{\rm I\!R}^n}\int_{{\rm I\!R}^n} \frac{1}{|x|}|\xi|^2 \sin^2(\theta(x,\xi)) f(x+x_{0},\xi, t)\ dxd\xi  
\end{multline*}
\noindent The positive derivative appearing in the equation above is the time derivative of a bounded quantity: 
\begin{multline*}
    \big|\int_{{\rm I\!R}^n}\int_{{\rm I\!R}^n}   f(x+x_0,\xi,t)\frac{(x\cdot\xi)}{|x|}\ dxd\xi\big| \leq  \int_{{\rm I\!R}^n}\int_{{\rm I\!R}^n}   f(x+x_0,\xi,t)|\xi|\ dxd\xi \\ \leq \int_{{\rm I\!R}^n}\int_{{\rm I\!R}^n}   f(x,\xi,t)(1+|\xi|^2)\ dxd\xi\leq  {\bf M}+{\bf E} <\infty\\ 
\end{multline*}
Positivity of this derivative implies that for an arbitrary $x_{0} \in {\rm I\!R}^n$ the following limit exists:
\begin{align*}
\lim_{t\rightarrow \infty} \int_{{\rm I\!R}^n}\int_{{\rm I\!R}^n} f(x+x_0,\xi,t)\frac{(x\cdot\xi)}{|x|}\ dxd\xi < \infty
\end{align*}
Furthermore the previous argument indicates that:
\begin{multline}
    \int_{0}^{\infty}\int_{{\rm I\!R}^n}\int_{{\rm I\!R}^n} \frac{1}{|x|}|\xi|^2  \sin^2(\theta(x,\xi)) f(x+x_{0},\xi, t)\ dxd\xi dt \\ = \int_{0}^{\infty}\int_{{\rm I\!R}^n}\int_{{\rm I\!R}^n} \frac{1}{|x-x_0|}|\xi|^2 \sin^2(\theta(x-x_0,\xi)) f(x,\xi, t)\ dxd\xi dt < \infty
\end{multline}
\noindent  Recall Definition \ref{bcone} of a blind cone $C_{x_0}(x,c)$ and let $B(x_{0}, R) \subset {\rm I\!R}^n$ be the ball of radius $R$ centered at $x_{0}$. We will integrate the energy over the spatial variable inside the ball and over the space of velocities outside of the blind cones. It follows from (2.3) that: 
\begin{multline*}
    \sin^2(c)\frac{1}{R}\int_{0}^{\infty}\int_{{\rm I\!R}^n-C_{x_{0}}(x,c)}\int_{B(x_0, R)} f(x,\xi, t) |\xi|^2 \ dxd\xi dt \\ <  \int_{0}^{\infty}\int_{{\rm I\!R}^n}\int_{{\rm I\!R}^n} \frac{1}{|x-x_0|}|\xi|^2 \sin^2(\theta(x-x_0,\xi)) f(x,\xi, t) \ dxd\xi dt< \infty
\end{multline*}
The previous computation implies that for any $x_{0} \in {\rm I\!R}^n$ and $0<c$, the total energy contained in $B(x_0, R) \times {\rm I\!R}^n$ is integrable over time: 
\begin{align}
    \int_{0}^{\infty}\int_{{\rm I\!R}^n-C_{x_{0}}(x,c)}\int_{B(x_0, R)} |\xi|^2  f(x,\xi, t) \ dxd\xi dt < \infty
\end{align} \\ 
\begin{figure}[t]
    \centering
   
    \label{fig:my_label}
\tikzset{every picture/.style={line width=0.75pt}} 

\begin{tikzpicture}[x=0.75pt,y=0.75pt,yscale=-1,xscale=1]

\draw    (264,56) -- (482,184) ;
\draw   (196,177.75) .. controls (196,98.91) and (259.91,35) .. (338.75,35) .. controls (417.59,35) and (481.5,98.91) .. (481.5,177.75) .. controls (481.5,256.59) and (417.59,320.5) .. (338.75,320.5) .. controls (259.91,320.5) and (196,256.59) .. (196,177.75) -- cycle ;
\draw    (422,62) -- (298,314) ;
\draw  [dash pattern={on 4.5pt off 4.5pt}]  (390,132.27) -- (435,181) ;
\draw  [dash pattern={on 4.5pt off 4.5pt}]  (390,132.27) -- (390,201) ;
\draw  [dash pattern={on 4.5pt off 4.5pt}]  (387.85,125.73) -- (342,72) ;
\draw  [dash pattern={on 4.5pt off 4.5pt}]  (387.85,125.73) -- (390,60) ;
\draw    (348.15,225.28) .. controls (398.17,229.29) and (444.59,215.14) .. (468.23,187.18) ;
\draw [shift={(470,185)}, rotate = 488.07] [fill={rgb, 255:red, 0; green, 0; blue, 0 }  ][line width=0.08]  [draw opacity=0] (5.36,-2.57) -- (0,0) -- (5.36,2.57) -- cycle    ;
\draw [shift={(345,225)}, rotate = 5.44] [fill={rgb, 255:red, 0; green, 0; blue, 0 }  ][line width=0.08]  [draw opacity=0] (5.36,-2.57) -- (0,0) -- (5.36,2.57) -- cycle    ;
\draw  [dash pattern={on 0.84pt off 2.51pt}]  (354,36) -- (387.85,125.73) ;
\draw    (397.08,180.34) .. controls (404.82,178.5) and (411.13,175.58) .. (416.47,172.51) ;
\draw [shift={(419,171)}, rotate = 508.39] [fill={rgb, 255:red, 0; green, 0; blue, 0 }  ][line width=0.08]  [draw opacity=0] (3.57,-1.72) -- (0,0) -- (3.57,1.72) -- cycle    ;
\draw [shift={(394,181)}, rotate = 349.22] [fill={rgb, 255:red, 0; green, 0; blue, 0 }  ][line width=0.08]  [draw opacity=0] (3.57,-1.72) -- (0,0) -- (3.57,1.72) -- cycle    ;
\draw  [dash pattern={on 4.5pt off 4.5pt}]  (385.4,131.56) -- (338.6,158.59) ;
\draw  [dash pattern={on 4.5pt off 4.5pt}]  (384.4,128.56) -- (329.65,118.35) ;
\draw  [dash pattern={on 4.5pt off 4.5pt}]  (392.2,127.49) -- (441.4,99.41) ;
\draw  [dash pattern={on 4.5pt off 4.5pt}]  (392.2,129.49) -- (449.35,138.65) ;
\draw  [dash pattern={on 0.84pt off 2.51pt}]  (197,161) -- (382,131) ;
\draw  [dash pattern={on 0.84pt off 2.51pt}]  (336,182) -- (322,316) ;
\draw [shift={(322,316)}, rotate = 275.96] [color={rgb, 255:red, 0; green, 0; blue, 0 }  ][line width=0.75]      (5.59,-5.59) .. controls (2.5,-5.59) and (0,-3.09) .. (0,0) .. controls (0,3.09) and (2.5,5.59) .. (5.59,5.59) ;
\draw [shift={(336,182)}, rotate = 95.96] [color={rgb, 255:red, 0; green, 0; blue, 0 }  ][line width=0.75]      (5.59,-5.59) .. controls (2.5,-5.59) and (0,-3.09) .. (0,0) .. controls (0,3.09) and (2.5,5.59) .. (5.59,5.59) ;
\draw  [draw opacity=0][dash pattern={on 4.5pt off 4.5pt}] (559.05,232.58) .. controls (530.84,261.59) and (496.6,284.8) .. (458.31,300.23) -- (348,31.25) -- cycle ; \draw  [dash pattern={on 4.5pt off 4.5pt}] (559.05,232.58) .. controls (530.84,261.59) and (496.6,284.8) .. (458.31,300.23) ;
\draw  [draw opacity=0] (451.88,302.73) .. controls (419.6,314.86) and (384.58,321.5) .. (348,321.5) .. controls (300.19,321.5) and (255.05,310.15) .. (215.18,290.03) -- (348,31.25) -- cycle ; \draw   (451.88,302.73) .. controls (419.6,314.86) and (384.58,321.5) .. (348,321.5) .. controls (300.19,321.5) and (255.05,310.15) .. (215.18,290.03) ;
\draw  [draw opacity=0][dash pattern={on 4.5pt off 4.5pt}] (211.92,288.36) .. controls (177.71,270.56) and (147.49,246.26) .. (122.94,217.11) -- (348,31.25) -- cycle ; \draw  [dash pattern={on 4.5pt off 4.5pt}] (211.92,288.36) .. controls (177.71,270.56) and (147.49,246.26) .. (122.94,217.11) ;

\draw  [fill={rgb, 255:red, 255; green, 255; blue, 255 }  ,fill opacity=1 ] (319,320) .. controls (319,318.34) and (320.34,317) .. (322,317) .. controls (323.66,317) and (325,318.34) .. (325,320) .. controls (325,321.66) and (323.66,323) .. (322,323) .. controls (320.34,323) and (319,321.66) .. (319,320) -- cycle ;
\draw  [fill={rgb, 255:red, 255; green, 255; blue, 255 }  ,fill opacity=1 ] (333.75,177.75) .. controls (333.75,176.09) and (335.09,174.75) .. (336.75,174.75) .. controls (338.41,174.75) and (339.75,176.09) .. (339.75,177.75) .. controls (339.75,179.41) and (338.41,180.75) .. (336.75,180.75) .. controls (335.09,180.75) and (333.75,179.41) .. (333.75,177.75) -- cycle ;
\draw  [fill={rgb, 255:red, 255; green, 255; blue, 255 }  ,fill opacity=1 ] (194,161) .. controls (194,159.34) and (195.34,158) .. (197,158) .. controls (198.66,158) and (200,159.34) .. (200,161) .. controls (200,162.66) and (198.66,164) .. (197,164) .. controls (195.34,164) and (194,162.66) .. (194,161) -- cycle ;
\draw  [fill={rgb, 255:red, 255; green, 255; blue, 255 }  ,fill opacity=1 ] (351,36) .. controls (351,34.34) and (352.34,33) .. (354,33) .. controls (355.66,33) and (357,34.34) .. (357,36) .. controls (357,37.66) and (355.66,39) .. (354,39) .. controls (352.34,39) and (351,37.66) .. (351,36) -- cycle ;
\draw  [fill={rgb, 255:red, 255; green, 255; blue, 255 }  ,fill opacity=1 ] (474,215) .. controls (474,213.34) and (475.34,212) .. (477,212) .. controls (478.66,212) and (480,213.34) .. (480,215) .. controls (480,216.66) and (478.66,218) .. (477,218) .. controls (475.34,218) and (474,216.66) .. (474,215) -- cycle ;

\draw (339,16) node [anchor=north west][inner sep=0.75pt]    {$O_{1}$};
\draw (173,158) node [anchor=north west][inner sep=0.75pt]    {$O_{3}$};
\draw (476,220) node [anchor=north west][inner sep=0.75pt]    {$O_{2}$};
\draw (405,184) node [anchor=north west][inner sep=0.75pt]    {$c$};
\draw (410,233) node [anchor=north west][inner sep=0.75pt]    {$2c$};
\draw (311,222) node [anchor=north west][inner sep=0.75pt]    {$R$};

\end{tikzpicture}
\begin{adjustwidth}{65pt}{65pt}
\caption{\small A possible configuration of the three observers. This drawing includes blind cones $C_{O_1}(x,c)$,$C_{O_1}(x,2c)$, $C_{O_3}(x,c)$ and ball $B(0,R)$ and a fragment of $B(O_1,2R$). Here $O_{3} \notin P.$}  
\end{adjustwidth}
\end{figure}
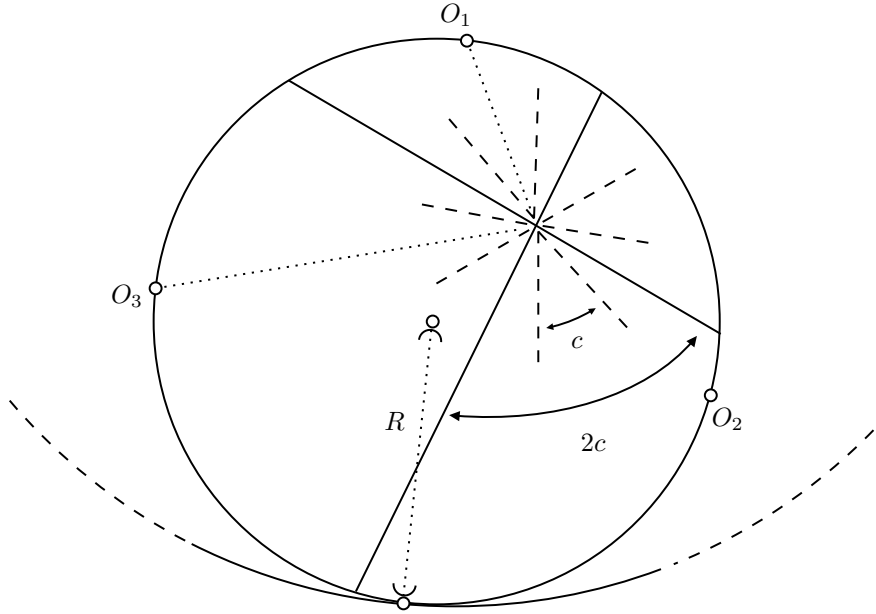
\noindent We will use the estimate above in the context of a geometric argument that is independent of the dimension and complete the proof. Choose any three distinct {\it observers}: $O_{1}, O_{2}, O_{3} \in \partial B(0, R)$ and for an arbitrary $x\in B(0, R)$ define $P=\partial B(0, R)\cap C_{O_1}(x, 2c)$. If there exists an $O_i$ such that $O_{i} \notin P$ then the blind cones $C_{O_i}(x, c)$ and $C_{O_1}(x, c)$ have an empty intersection. Set $P$ belongs to $\partial B(0, R)$ and is made of two path connected components, consider the longest short path on each component and set $K$ to be the maximum length of the two. For any fixed $R$ it is possible to choose $c$ small enough such that $K$ becomes as small as desired. Now set $c$ small enough such that $K$ becomes smaller than the shortest path on the sphere between any two of the there observers. The pigeon hole principle implies that, since each path connected component of $P$ can only contain maximum one of the observers, there exists an $O_{i}$ such that $O_{i} \notin P$. The previous argument implies that for any $x \in B(0, R)$ the blind cones with respect to the three observers have an empty intersection:   
\begin{align}
C_{O_1}(x,c) \cap C_{O_1}(x,c) \cap C_{O_1}(x,c)=\emptyset \end{align}
\noindent Continue with the following definition for $J_1, J_2$ and $J_3$. These integrals are convergent as a  consequence of (2.4):
\begin{align*}
    J_1=\int_{0}^{\infty}\int_{{\rm I\!R}^n-C_{O_{1}}(x, c)}\int_{B(O_1, 2R)} |\xi|^2  f(x,\xi, t) \ dxd\xi dt < \infty \\  J_2=\int_{0}^{\infty}\int_{{\rm I\!R}^n-C_{O_{2}(x, c)}}\int_{B(O_2, 2R)} |\xi|^2  f(x,\xi, t) \ dxd\xi dt < \infty \\  J_3=\int_{0}^{\infty}\int_{{\rm I\!R}^n-C_{O_{3}(x, c)}}\int_{B(O_3, 2R)} |\xi|^2  f(x,\xi, t) \ dxd\xi dt < \infty 
\end{align*}

\noindent Note that $B(0,R) \subset B(O_{1}, 2R) \cap B(O_{2}, 2R) \cap B(O_{3}, 2R)$. From (2.5) we know that for any $x \in B(0, R)$ the blind cones at $x$ with respect to the three observers intersect trivially. This shows that any subset of $B(0, R) \times {\rm I\!R}^n$ is covered at least  once in the domains of integration for $J_1, J_2$ and $J_3$. Therefore the positivity of integrands implies:  
\begin{align*}
    \int_{0}^{\infty}\int_{{\rm I\!R}^n}\int_{B(0, R)} f(x, \xi, t)|\xi|^2 \ dxd\xi dt < J_1+J_2+J_3 < \infty
\end{align*}
\noindent Finally, this argument can be replicated for a ball centered at an arbitrary point. Accordingly, because any bounded set is contained within a ball of finite radius the proof is complete. 

\end{proof}
\begin{corollary}\label{dispersionM} As a result of the previous theorem we have that the total mass of particles with magnitude of velocity greater or equal to any fixed positive number like $v$ is integrable within any bounded subsets of the spatial variable like $D \subset {{\rm I\!R}^n}$. Equivalently we have: 
\begin{align*}
    \int_{0}^{\infty}\int_{{\rm I\!R}^n - B(0,v)}\int_{D} f(x,\xi, t) \ dxd\xi dt < \infty
\end{align*}
\begin{proof}
    We will conclude this corollary from Theorem \ref{dispersionE} by imposing an arbitrary non zero lower bound like $v$ for $|\xi|$. Since $v$ is positive and arbitrarily small, removing $B(0, v)$ can be interpreted as removing the idle particles. Any non idle particle will eventually leave $D$ and therefore their total mass is integrable over time within the bounded set. 
\end{proof}
\end{corollary}
\begin{remark} The quantity below represents localization of mass for solutions of equations (1.1) at time $t$:
\begin{align*}
\int_{{\rm I\!R}^n}\int_{{\rm I\!R}^n} f(x,\xi,t)|x|^2 \ dxd\xi
\end{align*}
It is possible to show this quantity is convex over time using a similar integration by parts as proof of Theorem \ref{dispersionE},  simultaneously one could use Lemma \ref{consvlemma} to obtain an identical result. It follows that: 
\begin{align*}
    \partial^{2}_{t} \int_{{\rm I\!R}^n}\int_{{\rm I\!R}^n} f(x,\xi,t)|x|^2\ dxd\xi = 2 {\bf E}
\end{align*}
Although it is not possible to use this convex quantity to get similar outcome as the previous theorem, for the most part because this quantity is not priori bounded. The monotonicity of the angular momentum have the similar shortcoming. However the re-scaling which appears in the beginning of the proof of Theorem \ref{dispersionE} as localized angular momentum, will preserves the monotonicity while providing boundedness. In addition, any two distinct observers will have overlapping blind cones at points which are distant enough from both of them. Therefore, it is not possible to cover an unbounded set with finitely many such observers as before. This implies that the argument of the previous theorem is not useful in the case of an unbounded set.
\end{remark}

The previous theorem and corollary are proved independent of the specific structure of the interactions. Including more priori assumptions about the solutions of equation (1.1), paired with Theorem \ref{dispersionE} or Corollary \ref{dispersionM} can illuminate the use of the term dispersion. For example, assume that $f$ belongs to a class of solutions which are bounded and have bounded derivatives with respect to time. Let $D \subset R^{n}$ be an arbitrary bounded set and $v$ any positive number. It follows from Corollary \ref{dispersionM} that $f(x, \xi, t)$ converges to zero as time goes to infinity for almost every $(x,\xi) \in D \times ({\rm I\!R}^n - B(0, v))$:
\begin{align*}
    \lim_{t\rightarrow \infty}f(x,\xi, t)=0
\end{align*}
\indent The previous argument indicates that inside any bounded set of the spatial variable like $D$, as time goes to infinity, almost every particle with a magnitude of velocity greater than $v$ will inevitably leave the bounded set. The previous assertion is equivalent to the statement that only the particles with an arbitrarily small velocity can remain inside $D$. Thus, as time goes to infinity, an idle observer will almost only identify particles with velocity $\xi = 0$ inside any bounded set. One can continue using the Galilean invariance of the setting. Assume that the same idle observer starts to move with a  non zero arbitrary constant velocity like $\xi$. As a consequence, when time goes to infinity, inside any bounded moving region with the same constant velocity, like $D + t\xi \subset {\rm I\!R}^n$, the particles with velocity $\xi$ are being observed as motionless with respect to the moving observer, while almost every particle with a different velocity will vanish from the moving region. However, this decay may or may not be uniform. This shows that, in a way, the different velocities or frequencies separate from each other along the evolution of the solutions to equation (1.1). Each particle has a tendency to travel with other particles whose velocities are indistinguishable from each other and as a result avoid interaction with particles that move at different velocities. 
\begin{corollary}
As a consequence of the Galilean invariance for the solutions of the equation (1.1), Theorem \ref{dispersionE} and Corollary \ref{dispersionM} can be replicated for a bounded region that moves with constant velocity $\xi$.
\end{corollary}

The previous discussion is only one route to demonstrate the universality of dispersion along the evolution of particles. One can support the idea of dispersion, even without assumptions on the boundedness of solutions or their derivatives. The rest of this section is faithful to this idea. Recall definition of a blind cone $C_{x_0}(x,c)$ at $x$ with respect to observer $x_0$. We will modify this definition by allowing velocities that their magnitude is smaller than an arbitrary fixed  positive number. 
 
\begin{definition}\label{pbcone} Let for $K_{x_0}(x, c, v)$ be the {\it punctured blind cone} at $x \in {\rm I\!R}^n$ with respect to the observer $x_0 \in {\rm I\!R}^n$ for any positive $c$ and $v$ by modifying definition of a blind cone (Definition \ref{bcone}): 
\begin{align*}
    K_{x_{0}}(x, c, v) = C_{x_{0}}(x,c) \cup B(0, v)    
\end{align*}
\end{definition}

\begin{definition}\label{conecollection}
Let $\Gamma_{x_{0}}(c, v)$ be the {\it collection of punctured blind cones} over the spatial variable, relative to $x_0 \in {\rm I\!R}^n$ for $0< c, v$: 
\begin{align*}
    \Gamma_{x_{0}}(c,v)=\{(x,\xi) \in {\rm I\!R}^n \times {\rm I\!R}^n | \ \xi \in K_{x_{0}}(x, c, v)\} 
\end{align*}

\end{definition}
\begin{theorem}\label{concentration} For any solution of equation (1.1) satisfying (2.1) and (2.2), the total amount of mass(${\bf M}$) concentrates over time within the collection of punctured blind cones in the sense defined below. For any positive $c$ and $v$ and an arbitrary observer $x_{0} \in{\rm I\!R}^n$, we have: \begin{align*}
 \lim_{T\rightarrow \infty} \frac{1}{T}\int_{0}^{T }\iint\limits_{\Gamma_{x_{0}}(c,v)} f(x, \xi, t) \ dxd\xi dt = {\bf M}
\end{align*}
\end{theorem}
\begin{proof}
Define the following two sets for any $0<R$ and $x_0 \in {\rm I\!R}^n$: 
\begin{align*}
    N_R(x_0) =\{ (x,\xi)\in {\rm I\!R}^n \times ({\rm I\!R}^n - K_{x_0}(x,c,v)) \ |\ \  |x|>R\} \\ 
  M_R(x_0) =\{ (x,\xi)\in {\rm I\!R}^n \times ({\rm I\!R}^n - K_{x_0}(x,c,v)) \ |\ \  |x|\leq R\}
\end{align*}

\tikzset{every picture/.style={line width=0.75pt}} 
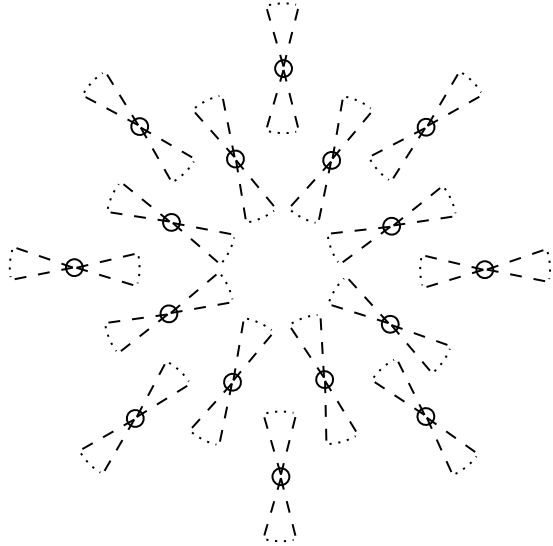
\begin{figure}[t]
    \centering
\begin{tikzpicture}[x=0.75pt,y=0.75pt,yscale=-1,xscale=1]

\draw [color={rgb, 255:red, 0; green, 0; blue, 0 }  ,draw opacity=1 ] [dash pattern={on 4.5pt off 4.5pt}]  (399.15,146.35) -- (394.08,176.09) ;
\draw [color={rgb, 255:red, 0; green, 0; blue, 0 }  ,draw opacity=1 ] [dash pattern={on 4.5pt off 4.5pt}]  (394.08,176.09) -- (414.14,152.12) ;
\draw [color={rgb, 255:red, 0; green, 0; blue, 0 }  ,draw opacity=1 ] [dash pattern={on 4.5pt off 4.5pt}]  (373.18,200.04) -- (393.34,177.61) ;
\draw [color={rgb, 255:red, 0; green, 0; blue, 0 }  ,draw opacity=1 ] [dash pattern={on 4.5pt off 4.5pt}]  (393.34,177.61) -- (387,208.21) ;
\draw  [color={rgb, 255:red, 0; green, 0; blue, 0 }  ,draw opacity=1 ] (389.88,175.01) .. controls (390.97,172.93) and (393.54,172.13) .. (395.62,173.23) .. controls (397.7,174.32) and (398.5,176.89) .. (397.41,178.96) .. controls (396.32,181.04) and (393.75,181.84) .. (391.67,180.75) .. controls (389.59,179.66) and (388.79,177.09) .. (389.88,175.01) -- cycle ;

\draw  [draw opacity=0][dash pattern={on 0.84pt off 2.51pt}] (373.19,202.32) .. controls (375.91,204.57) and (379.07,206.38) .. (382.59,207.62) .. controls (383.56,207.96) and (384.52,208.24) .. (385.49,208.48) -- (393.34,177.61) -- cycle ; \draw  [color={rgb, 255:red, 0; green, 0; blue, 0 }  ,draw opacity=1 ][dash pattern={on 0.84pt off 2.51pt}] (373.19,202.32) .. controls (375.91,204.57) and (379.07,206.38) .. (382.59,207.62) .. controls (383.56,207.96) and (384.52,208.24) .. (385.49,208.48) ;

\draw  [draw opacity=0][dash pattern={on 0.84pt off 2.51pt}] (400.97,144.95) .. controls (404.42,145.7) and (407.8,147.05) .. (410.96,149.05) .. controls (411.8,149.58) and (412.61,150.14) .. (413.38,150.74) -- (394.08,176.09) -- cycle ; \draw  [color={rgb, 255:red, 0; green, 0; blue, 0 }  ,draw opacity=1 ][dash pattern={on 0.84pt off 2.51pt}] (400.97,144.95) .. controls (404.42,145.7) and (407.8,147.05) .. (410.96,149.05) .. controls (411.8,149.58) and (412.61,150.14) .. (413.38,150.74) ;

\draw [color={rgb, 255:red, 0; green, 0; blue, 0 }  ,draw opacity=1 ] [dash pattern={on 4.5pt off 4.5pt}]  (361,100.67) -- (369.38,129.64) ;
\draw [color={rgb, 255:red, 0; green, 0; blue, 0 }  ,draw opacity=1 ] [dash pattern={on 4.5pt off 4.5pt}]  (369.38,129.64) -- (377,99.33) ;
\draw [color={rgb, 255:red, 0; green, 0; blue, 0 }  ,draw opacity=1 ] [dash pattern={on 4.5pt off 4.5pt}]  (361,160.31) -- (369.38,131.33) ;
\draw [color={rgb, 255:red, 0; green, 0; blue, 0 }  ,draw opacity=1 ] [dash pattern={on 4.5pt off 4.5pt}]  (369.38,131.33) -- (377,161.64) ;
\draw  [color={rgb, 255:red, 0; green, 0; blue, 0 }  ,draw opacity=1 ] (365.14,130.5) .. controls (365.21,128.16) and (367.18,126.32) .. (369.52,126.39) .. controls (371.87,126.47) and (373.71,128.44) .. (373.63,130.78) .. controls (373.55,133.13) and (371.59,134.97) .. (369.24,134.89) .. controls (366.9,134.81) and (365.06,132.85) .. (365.14,130.5) -- cycle ;

\draw  [draw opacity=0][dash pattern={on 0.84pt off 2.51pt}] (362,162.36) .. controls (365.43,163.19) and (369.06,163.45) .. (372.78,163.03) .. controls (373.79,162.91) and (374.79,162.75) .. (375.76,162.54) -- (369.38,131.33) -- cycle ; \draw  [color={rgb, 255:red, 0; green, 0; blue, 0 }  ,draw opacity=1 ][dash pattern={on 0.84pt off 2.51pt}] (362,162.36) .. controls (365.43,163.19) and (369.06,163.45) .. (372.78,163.03) .. controls (373.79,162.91) and (374.79,162.75) .. (375.76,162.54) ;

\draw  [draw opacity=0][dash pattern={on 0.84pt off 2.51pt}] (362.02,98.61) .. controls (365.46,97.78) and (369.09,97.53) .. (372.8,97.95) .. controls (373.79,98.06) and (374.76,98.22) .. (375.71,98.42) -- (369.38,129.64) -- cycle ; \draw  [color={rgb, 255:red, 0; green, 0; blue, 0 }  ,draw opacity=1 ][dash pattern={on 0.84pt off 2.51pt}] (362.02,98.61) .. controls (365.46,97.78) and (369.09,97.53) .. (372.8,97.95) .. controls (373.79,98.06) and (374.76,98.22) .. (375.71,98.42) ;

\draw [color={rgb, 255:red, 0; green, 0; blue, 0 }  ,draw opacity=1 ] [dash pattern={on 4.5pt off 4.5pt}]  (269.51,145.01) -- (296.12,159.23) ;
\draw [color={rgb, 255:red, 0; green, 0; blue, 0 }  ,draw opacity=1 ] [dash pattern={on 4.5pt off 4.5pt}]  (296.12,159.23) -- (279.72,132.62) ;
\draw [color={rgb, 255:red, 0; green, 0; blue, 0 }  ,draw opacity=1 ] [dash pattern={on 4.5pt off 4.5pt}]  (312.23,186.63) -- (297.33,160.41) ;
\draw [color={rgb, 255:red, 0; green, 0; blue, 0 }  ,draw opacity=1 ] [dash pattern={on 4.5pt off 4.5pt}]  (297.33,160.41) -- (324.35,176.1) ;
\draw  [color={rgb, 255:red, 0; green, 0; blue, 0 }  ,draw opacity=1 ] (293.77,162.87) .. controls (292.14,161.18) and (292.2,158.49) .. (293.89,156.86) .. controls (295.58,155.24) and (298.27,155.29) .. (299.9,156.98) .. controls (301.53,158.67) and (301.47,161.36) .. (299.78,162.99) .. controls (298.09,164.62) and (295.4,164.56) .. (293.77,162.87) -- cycle ;

\draw  [draw opacity=0][dash pattern={on 0.84pt off 2.51pt}] (314.4,187.35) .. controls (317.39,185.47) and (320.11,183.04) .. (322.4,180.09) .. controls (323.02,179.29) and (323.6,178.46) .. (324.14,177.62) -- (297.33,160.41) -- cycle ; \draw  [color={rgb, 255:red, 0; green, 0; blue, 0 }  ,draw opacity=1 ][dash pattern={on 0.84pt off 2.51pt}] (314.4,187.35) .. controls (317.39,185.47) and (320.11,183.04) .. (322.4,180.09) .. controls (323.02,179.29) and (323.6,178.46) .. (324.14,177.62) ;

\draw  [draw opacity=0][dash pattern={on 0.84pt off 2.51pt}] (268.75,142.85) .. controls (270.56,139.81) and (272.91,137.03) .. (275.8,134.67) .. controls (276.57,134.04) and (277.36,133.45) .. (278.17,132.91) -- (296.12,159.23) -- cycle ; \draw  [color={rgb, 255:red, 0; green, 0; blue, 0 }  ,draw opacity=1 ][dash pattern={on 0.84pt off 2.51pt}] (268.75,142.85) .. controls (270.56,139.81) and (272.91,137.03) .. (275.8,134.67) .. controls (276.57,134.04) and (277.36,133.45) .. (278.17,132.91) ;

\draw [color={rgb, 255:red, 0; green, 0; blue, 0 }  ,draw opacity=1 ] [dash pattern={on 4.5pt off 4.5pt}]  (294.24,224.26) -- (264.89,231.19) ;
\draw [color={rgb, 255:red, 0; green, 0; blue, 0 }  ,draw opacity=1 ] [dash pattern={on 4.5pt off 4.5pt}]  (264.89,231.19) -- (294.78,240.31) ;
\draw [color={rgb, 255:red, 0; green, 0; blue, 0 }  ,draw opacity=1 ] [dash pattern={on 4.5pt off 4.5pt}]  (234.68,221.28) -- (263.2,231.1) ;
\draw [color={rgb, 255:red, 0; green, 0; blue, 0 }  ,draw opacity=1 ] [dash pattern={on 4.5pt off 4.5pt}]  (263.2,231.1) -- (232.55,237.2) ;
\draw  [color={rgb, 255:red, 0; green, 0; blue, 0 }  ,draw opacity=1 ] (264.24,226.9) .. controls (266.58,227.1) and (268.32,229.15) .. (268.12,231.49) .. controls (267.93,233.83) and (265.87,235.57) .. (263.54,235.37) .. controls (261.2,235.18) and (259.46,233.12) .. (259.65,230.79) .. controls (259.85,228.45) and (261.9,226.71) .. (264.24,226.9) -- cycle ;

\draw  [draw opacity=0][dash pattern={on 0.84pt off 2.51pt}] (232.58,222.18) .. controls (231.57,225.57) and (231.14,229.18) .. (231.37,232.91) .. controls (231.44,233.93) and (231.55,234.93) .. (231.71,235.92) -- (263.2,231.1) -- cycle ; \draw  [color={rgb, 255:red, 0; green, 0; blue, 0 }  ,draw opacity=1 ][dash pattern={on 0.84pt off 2.51pt}] (232.58,222.18) .. controls (231.57,225.57) and (231.14,229.18) .. (231.37,232.91) .. controls (231.44,233.93) and (231.55,234.93) .. (231.71,235.92) ;

\draw  [draw opacity=0][dash pattern={on 0.84pt off 2.51pt}] (296.25,225.38) .. controls (296.9,228.85) and (296.97,232.49) .. (296.37,236.18) .. controls (296.21,237.16) and (296,238.12) .. (295.75,239.07) -- (264.89,231.19) -- cycle ; \draw  [color={rgb, 255:red, 0; green, 0; blue, 0 }  ,draw opacity=1 ][dash pattern={on 0.84pt off 2.51pt}] (296.25,225.38) .. controls (296.9,228.85) and (296.97,232.49) .. (296.37,236.18) .. controls (296.21,237.16) and (296,238.12) .. (295.75,239.07) ;

\draw [color={rgb, 255:red, 0; green, 0; blue, 0 }  ,draw opacity=1 ] [dash pattern={on 4.5pt off 4.5pt}]  (309.48,279.92) -- (295.3,306.55) ;
\draw [color={rgb, 255:red, 0; green, 0; blue, 0 }  ,draw opacity=1 ] [dash pattern={on 4.5pt off 4.5pt}]  (295.3,306.55) -- (321.88,290.11) ;
\draw [color={rgb, 255:red, 0; green, 0; blue, 0 }  ,draw opacity=1 ] [dash pattern={on 4.5pt off 4.5pt}]  (267.92,322.7) -- (294.12,307.76) ;
\draw [color={rgb, 255:red, 0; green, 0; blue, 0 }  ,draw opacity=1 ] [dash pattern={on 4.5pt off 4.5pt}]  (294.12,307.76) -- (278.47,334.8) ;
\draw  [color={rgb, 255:red, 0; green, 0; blue, 0 }  ,draw opacity=1 ] (291.65,304.2) .. controls (293.34,302.57) and (296.03,302.62) .. (297.66,304.31) .. controls (299.29,306) and (299.24,308.69) .. (297.55,310.32) .. controls (295.86,311.95) and (293.17,311.9) .. (291.54,310.21) .. controls (289.91,308.52) and (289.96,305.83) .. (291.65,304.2) -- cycle ;

\draw  [draw opacity=0][dash pattern={on 0.84pt off 2.51pt}] (267.21,324.87) .. controls (269.09,327.86) and (271.52,330.57) .. (274.47,332.86) .. controls (275.28,333.48) and (276.1,334.06) .. (276.95,334.59) -- (294.12,307.76) -- cycle ; \draw  [color={rgb, 255:red, 0; green, 0; blue, 0 }  ,draw opacity=1 ][dash pattern={on 0.84pt off 2.51pt}] (267.21,324.87) .. controls (269.09,327.86) and (271.52,330.57) .. (274.47,332.86) .. controls (275.28,333.48) and (276.1,334.06) .. (276.95,334.59) ;

\draw  [draw opacity=0][dash pattern={on 0.84pt off 2.51pt}] (311.64,279.16) .. controls (314.68,280.96) and (317.46,283.3) .. (319.83,286.19) .. controls (320.46,286.96) and (321.05,287.75) .. (321.59,288.56) -- (295.3,306.55) -- cycle ; \draw  [color={rgb, 255:red, 0; green, 0; blue, 0 }  ,draw opacity=1 ][dash pattern={on 0.84pt off 2.51pt}] (311.64,279.16) .. controls (314.68,280.96) and (317.46,283.3) .. (319.83,286.19) .. controls (320.46,286.96) and (321.05,287.75) .. (321.59,288.56) ;

\draw [color={rgb, 255:red, 0; green, 0; blue, 0 }  ,draw opacity=1 ] [dash pattern={on 4.5pt off 4.5pt}]  (359.67,306.67) -- (368.05,335.64) ;
\draw [color={rgb, 255:red, 0; green, 0; blue, 0 }  ,draw opacity=1 ] [dash pattern={on 4.5pt off 4.5pt}]  (368.05,335.64) -- (375.67,305.33) ;
\draw [color={rgb, 255:red, 0; green, 0; blue, 0 }  ,draw opacity=1 ] [dash pattern={on 4.5pt off 4.5pt}]  (359.67,366.31) -- (368.05,337.33) ;
\draw [color={rgb, 255:red, 0; green, 0; blue, 0 }  ,draw opacity=1 ] [dash pattern={on 4.5pt off 4.5pt}]  (368.05,337.33) -- (375.67,367.64) ;
\draw  [color={rgb, 255:red, 0; green, 0; blue, 0 }  ,draw opacity=1 ] (363.8,336.5) .. controls (363.88,334.16) and (365.84,332.32) .. (368.19,332.39) .. controls (370.54,332.47) and (372.38,334.44) .. (372.3,336.78) .. controls (372.22,339.13) and (370.26,340.97) .. (367.91,340.89) .. controls (365.57,340.81) and (363.73,338.85) .. (363.8,336.5) -- cycle ;

\draw  [draw opacity=0][dash pattern={on 0.84pt off 2.51pt}] (360.67,368.36) .. controls (364.1,369.19) and (367.73,369.45) .. (371.44,369.03) .. controls (372.46,368.91) and (373.45,368.75) .. (374.43,368.54) -- (368.05,337.33) -- cycle ; \draw  [color={rgb, 255:red, 0; green, 0; blue, 0 }  ,draw opacity=1 ][dash pattern={on 0.84pt off 2.51pt}] (360.67,368.36) .. controls (364.1,369.19) and (367.73,369.45) .. (371.44,369.03) .. controls (372.46,368.91) and (373.45,368.75) .. (374.43,368.54) ;

\draw  [draw opacity=0][dash pattern={on 0.84pt off 2.51pt}] (360.69,304.61) .. controls (364.12,303.78) and (367.75,303.53) .. (371.47,303.95) .. controls (372.45,304.06) and (373.43,304.22) .. (374.38,304.42) -- (368.05,335.64) -- cycle ; \draw  [color={rgb, 255:red, 0; green, 0; blue, 0 }  ,draw opacity=1 ][dash pattern={on 0.84pt off 2.51pt}] (360.69,304.61) .. controls (364.12,303.78) and (367.75,303.53) .. (371.47,303.95) .. controls (372.45,304.06) and (373.43,304.22) .. (374.38,304.42) ;

\draw [color={rgb, 255:red, 0; green, 0; blue, 0 }  ,draw opacity=1 ] [dash pattern={on 4.5pt off 4.5pt}]  (415.33,288.89) -- (440.51,305.51) ;
\draw [color={rgb, 255:red, 0; green, 0; blue, 0 }  ,draw opacity=1 ] [dash pattern={on 4.5pt off 4.5pt}]  (440.51,305.51) -- (426.65,277.5) ;
\draw [color={rgb, 255:red, 0; green, 0; blue, 0 }  ,draw opacity=1 ] [dash pattern={on 4.5pt off 4.5pt}]  (454.02,334.29) -- (441.61,306.8) ;
\draw [color={rgb, 255:red, 0; green, 0; blue, 0 }  ,draw opacity=1 ] [dash pattern={on 4.5pt off 4.5pt}]  (441.61,306.8) -- (467.06,324.92) ;
\draw  [color={rgb, 255:red, 0; green, 0; blue, 0 }  ,draw opacity=1 ] (437.83,308.92) .. controls (436.37,307.08) and (436.67,304.41) .. (438.51,302.95) .. controls (440.34,301.48) and (443.02,301.79) .. (444.48,303.62) .. controls (445.94,305.46) and (445.64,308.13) .. (443.81,309.59) .. controls (441.97,311.06) and (439.3,310.75) .. (437.83,308.92) -- cycle ;

\draw  [draw opacity=0][dash pattern={on 0.84pt off 2.51pt}] (456.11,335.2) .. controls (459.26,333.6) and (462.19,331.44) .. (464.75,328.72) .. controls (465.44,327.97) and (466.1,327.2) .. (466.71,326.41) -- (441.61,306.8) -- cycle ; \draw  [color={rgb, 255:red, 0; green, 0; blue, 0 }  ,draw opacity=1 ][dash pattern={on 0.84pt off 2.51pt}] (456.11,335.2) .. controls (459.26,333.6) and (462.19,331.44) .. (464.75,328.72) .. controls (465.44,327.97) and (466.1,327.2) .. (466.71,326.41) ;

\draw  [draw opacity=0][dash pattern={on 0.84pt off 2.51pt}] (414.78,286.66) .. controls (416.85,283.81) and (419.45,281.26) .. (422.55,279.17) .. controls (423.38,278.62) and (424.22,278.11) .. (425.08,277.64) -- (440.51,305.51) -- cycle ; \draw  [color={rgb, 255:red, 0; green, 0; blue, 0 }  ,draw opacity=1 ][dash pattern={on 0.84pt off 2.51pt}] (414.78,286.66) .. controls (416.85,283.81) and (419.45,281.26) .. (422.55,279.17) .. controls (423.38,278.62) and (424.22,278.11) .. (425.08,277.64) ;

\draw [color={rgb, 255:red, 0; green, 0; blue, 0 }  ,draw opacity=1 ] [dash pattern={on 4.5pt off 4.5pt}]  (500.42,222.32) -- (471.89,232.1) ;
\draw [color={rgb, 255:red, 0; green, 0; blue, 0 }  ,draw opacity=1 ] [dash pattern={on 4.5pt off 4.5pt}]  (471.89,232.1) -- (502.53,238.24) ;
\draw [color={rgb, 255:red, 0; green, 0; blue, 0 }  ,draw opacity=1 ] [dash pattern={on 4.5pt off 4.5pt}]  (440.85,225.22) -- (470.2,232.19) ;
\draw [color={rgb, 255:red, 0; green, 0; blue, 0 }  ,draw opacity=1 ] [dash pattern={on 4.5pt off 4.5pt}]  (470.2,232.19) -- (440.29,241.27) ;
\draw  [color={rgb, 255:red, 0; green, 0; blue, 0 }  ,draw opacity=1 ] (470.82,227.9) .. controls (473.17,227.87) and (475.1,229.74) .. (475.14,232.09) .. controls (475.17,234.43) and (473.3,236.36) .. (470.95,236.4) .. controls (468.61,236.44) and (466.67,234.57) .. (466.64,232.22) .. controls (466.6,229.87) and (468.47,227.94) .. (470.82,227.9) -- cycle ;

\draw  [draw opacity=0][dash pattern={on 0.84pt off 2.51pt}] (438.85,226.32) .. controls (438.18,229.79) and (438.11,233.43) .. (438.71,237.12) .. controls (438.87,238.12) and (439.08,239.11) .. (439.33,240.08) -- (470.2,232.19) -- cycle ; \draw  [color={rgb, 255:red, 0; green, 0; blue, 0 }  ,draw opacity=1 ][dash pattern={on 0.84pt off 2.51pt}] (438.85,226.32) .. controls (438.18,229.79) and (438.11,233.43) .. (438.71,237.12) .. controls (438.87,238.12) and (439.08,239.11) .. (439.33,240.08) ;

\draw  [draw opacity=0][dash pattern={on 0.84pt off 2.51pt}] (502.52,223.24) .. controls (503.52,226.63) and (503.95,230.24) .. (503.71,233.97) .. controls (503.64,234.96) and (503.53,235.94) .. (503.38,236.91) -- (471.89,232.1) -- cycle ; \draw  [color={rgb, 255:red, 0; green, 0; blue, 0 }  ,draw opacity=1 ][dash pattern={on 0.84pt off 2.51pt}] (502.52,223.24) .. controls (503.52,226.63) and (503.95,230.24) .. (503.71,233.97) .. controls (503.64,234.96) and (503.53,235.94) .. (503.38,236.91) ;

\draw [color={rgb, 255:red, 0; green, 0; blue, 0 }  ,draw opacity=1 ] [dash pattern={on 4.5pt off 4.5pt}]  (457.27,133.56) -- (441.99,159.57) ;
\draw [color={rgb, 255:red, 0; green, 0; blue, 0 }  ,draw opacity=1 ] [dash pattern={on 4.5pt off 4.5pt}]  (441.99,159.57) -- (469.23,144.26) ;
\draw [color={rgb, 255:red, 0; green, 0; blue, 0 }  ,draw opacity=1 ] [dash pattern={on 4.5pt off 4.5pt}]  (413.96,174.57) -- (440.76,160.73) ;
\draw [color={rgb, 255:red, 0; green, 0; blue, 0 }  ,draw opacity=1 ] [dash pattern={on 4.5pt off 4.5pt}]  (440.76,160.73) -- (423.99,187.1) ;
\draw  [color={rgb, 255:red, 0; green, 0; blue, 0 }  ,draw opacity=1 ] (438.45,157.08) .. controls (440.2,155.52) and (442.89,155.68) .. (444.44,157.44) .. controls (446,159.19) and (445.84,161.88) .. (444.08,163.44) .. controls (442.33,164.99) and (439.64,164.83) .. (438.08,163.08) .. controls (436.53,161.32) and (436.69,158.63) .. (438.45,157.08) -- cycle ;

\draw  [draw opacity=0][dash pattern={on 0.84pt off 2.51pt}] (413.16,176.71) .. controls (414.91,179.77) and (417.23,182.58) .. (420.08,184.99) .. controls (420.86,185.65) and (421.67,186.26) .. (422.49,186.83) -- (440.76,160.73) -- cycle ; \draw  [color={rgb, 255:red, 0; green, 0; blue, 0 }  ,draw opacity=1 ][dash pattern={on 0.84pt off 2.51pt}] (413.16,176.71) .. controls (414.91,179.77) and (417.23,182.58) .. (420.08,184.99) .. controls (420.86,185.65) and (421.67,186.26) .. (422.49,186.83) ;

\draw  [draw opacity=0][dash pattern={on 0.84pt off 2.51pt}] (459.46,132.89) .. controls (462.42,134.81) and (465.1,137.27) .. (467.35,140.26) .. controls (467.95,141.05) and (468.5,141.87) .. (469.01,142.7) -- (441.99,159.57) -- cycle ; \draw  [color={rgb, 255:red, 0; green, 0; blue, 0 }  ,draw opacity=1 ][dash pattern={on 0.84pt off 2.51pt}] (459.46,132.89) .. controls (462.42,134.81) and (465.1,137.27) .. (467.35,140.26) .. controls (467.95,141.05) and (468.5,141.87) .. (469.01,142.7) ;

\draw [color={rgb, 255:red, 0; green, 0; blue, 0 }  ,draw opacity=1 ] [dash pattern={on 4.5pt off 4.5pt}]  (324.69,152.81) -- (344.69,175.39) ;
\draw [color={rgb, 255:red, 0; green, 0; blue, 0 }  ,draw opacity=1 ] [dash pattern={on 4.5pt off 4.5pt}]  (344.69,175.39) -- (338.57,144.74) ;
\draw [color={rgb, 255:red, 0; green, 0; blue, 0 }  ,draw opacity=1 ] [dash pattern={on 4.5pt off 4.5pt}]  (350.27,206.69) -- (345.42,176.92) ;
\draw [color={rgb, 255:red, 0; green, 0; blue, 0 }  ,draw opacity=1 ] [dash pattern={on 4.5pt off 4.5pt}]  (345.42,176.92) -- (365.3,201.03) ;
\draw  [color={rgb, 255:red, 0; green, 0; blue, 0 }  ,draw opacity=1 ] (341.22,177.99) .. controls (340.29,175.84) and (341.27,173.33) .. (343.42,172.4) .. controls (345.58,171.46) and (348.08,172.44) .. (349.02,174.6) .. controls (349.95,176.75) and (348.97,179.25) .. (346.82,180.19) .. controls (344.66,181.13) and (342.16,180.14) .. (341.22,177.99) -- cycle ;

\draw  [draw opacity=0][dash pattern={on 0.84pt off 2.51pt}] (352.06,208.11) .. controls (355.51,207.39) and (358.9,206.06) .. (362.08,204.09) .. controls (362.94,203.55) and (363.77,202.98) .. (364.57,202.37) -- (345.42,176.92) -- cycle ; \draw  [color={rgb, 255:red, 0; green, 0; blue, 0 }  ,draw opacity=1 ][dash pattern={on 0.84pt off 2.51pt}] (352.06,208.11) .. controls (355.51,207.39) and (358.9,206.06) .. (362.08,204.09) .. controls (362.94,203.55) and (363.77,202.98) .. (364.57,202.37) ;

\draw  [draw opacity=0][dash pattern={on 0.84pt off 2.51pt}] (324.73,150.51) .. controls (327.48,148.29) and (330.65,146.51) .. (334.18,145.3) .. controls (335.12,144.97) and (336.07,144.7) .. (337.02,144.47) -- (344.69,175.39) -- cycle ; \draw  [color={rgb, 255:red, 0; green, 0; blue, 0 }  ,draw opacity=1 ][dash pattern={on 0.84pt off 2.51pt}] (324.73,150.51) .. controls (327.48,148.29) and (330.65,146.51) .. (334.18,145.3) .. controls (335.12,144.97) and (336.07,144.7) .. (337.02,144.47) ;

\draw [color={rgb, 255:red, 0; green, 0; blue, 0 }  ,draw opacity=1 ] [dash pattern={on 4.5pt off 4.5pt}]  (448.35,190.97) -- (424.82,209.84) ;
\draw [color={rgb, 255:red, 0; green, 0; blue, 0 }  ,draw opacity=1 ] [dash pattern={on 4.5pt off 4.5pt}]  (424.82,209.84) -- (455.73,205.23) ;
\draw [color={rgb, 255:red, 0; green, 0; blue, 0 }  ,draw opacity=1 ] [dash pattern={on 4.5pt off 4.5pt}]  (393.29,213.88) -- (423.26,210.49) ;
\draw [color={rgb, 255:red, 0; green, 0; blue, 0 }  ,draw opacity=1 ] [dash pattern={on 4.5pt off 4.5pt}]  (423.26,210.49) -- (398.2,229.16) ;
\draw  [color={rgb, 255:red, 0; green, 0; blue, 0 }  ,draw opacity=1 ] (422.4,206.25) .. controls (424.59,205.42) and (427.04,206.53) .. (427.87,208.72) .. controls (428.7,210.92) and (427.6,213.37) .. (425.4,214.2) .. controls (423.21,215.03) and (420.75,213.92) .. (419.92,211.73) .. controls (419.09,209.53) and (420.2,207.08) .. (422.4,206.25) -- cycle ;

\draw  [draw opacity=0][dash pattern={on 0.84pt off 2.51pt}] (391.78,215.59) .. controls (392.33,219.08) and (393.49,222.53) .. (395.3,225.8) .. controls (395.8,226.69) and (396.33,227.55) .. (396.9,228.37) -- (423.26,210.49) -- cycle ; \draw  [color={rgb, 255:red, 0; green, 0; blue, 0 }  ,draw opacity=1 ][dash pattern={on 0.84pt off 2.51pt}] (391.78,215.59) .. controls (392.33,219.08) and (393.49,222.53) .. (395.3,225.8) .. controls (395.8,226.69) and (396.33,227.55) .. (396.9,228.37) ;

\draw  [draw opacity=0][dash pattern={on 0.84pt off 2.51pt}] (450.65,191.12) .. controls (452.73,193.97) and (454.36,197.23) .. (455.39,200.82) .. controls (455.67,201.77) and (455.9,202.73) .. (456.08,203.69) -- (424.82,209.84) -- cycle ; \draw  [color={rgb, 255:red, 0; green, 0; blue, 0 }  ,draw opacity=1 ][dash pattern={on 0.84pt off 2.51pt}] (450.65,191.12) .. controls (452.73,193.97) and (454.36,197.23) .. (455.39,200.82) .. controls (455.67,201.77) and (455.9,202.73) .. (456.08,203.69) ;

\draw [color={rgb, 255:red, 0; green, 0; blue, 0 }  ,draw opacity=1 ] [dash pattern={on 4.5pt off 4.5pt}]  (393.59,250.08) -- (422.36,259.17) ;
\draw [color={rgb, 255:red, 0; green, 0; blue, 0 }  ,draw opacity=1 ] [dash pattern={on 4.5pt off 4.5pt}]  (422.36,259.17) -- (401.35,236.02) ;
\draw [color={rgb, 255:red, 0; green, 0; blue, 0 }  ,draw opacity=1 ] [dash pattern={on 4.5pt off 4.5pt}]  (443.23,283.15) -- (423.76,260.1) ;
\draw [color={rgb, 255:red, 0; green, 0; blue, 0 }  ,draw opacity=1 ] [dash pattern={on 4.5pt off 4.5pt}]  (423.76,260.1) -- (453.21,270.57) ;
\draw  [color={rgb, 255:red, 0; green, 0; blue, 0 }  ,draw opacity=1 ] (420.72,263.18) .. controls (418.81,261.81) and (418.36,259.16) .. (419.73,257.25) .. controls (421.09,255.34) and (423.75,254.9) .. (425.66,256.26) .. controls (427.57,257.63) and (428.01,260.28) .. (426.65,262.19) .. controls (425.28,264.1) and (422.63,264.54) .. (420.72,263.18) -- cycle ;

\draw  [draw opacity=0][dash pattern={on 0.84pt off 2.51pt}] (445.49,283.45) .. controls (448.09,281.05) and (450.31,278.17) .. (452.02,274.85) .. controls (452.49,273.94) and (452.9,273.03) .. (453.27,272.1) -- (423.76,260.1) -- cycle ; \draw  [color={rgb, 255:red, 0; green, 0; blue, 0 }  ,draw opacity=1 ][dash pattern={on 0.84pt off 2.51pt}] (445.49,283.45) .. controls (448.09,281.05) and (450.31,278.17) .. (452.02,274.85) .. controls (452.49,273.94) and (452.9,273.03) .. (453.27,272.1) ;

\draw  [draw opacity=0][dash pattern={on 0.84pt off 2.51pt}] (392.45,248.09) .. controls (393.66,244.77) and (395.46,241.61) .. (397.87,238.75) .. controls (398.51,237.99) and (399.19,237.27) .. (399.88,236.59) -- (422.36,259.17) -- cycle ; \draw  [color={rgb, 255:red, 0; green, 0; blue, 0 }  ,draw opacity=1 ][dash pattern={on 0.84pt off 2.51pt}] (392.45,248.09) .. controls (393.66,244.77) and (395.46,241.61) .. (397.87,238.75) .. controls (398.51,237.99) and (399.19,237.27) .. (399.88,236.59) ;

\draw [color={rgb, 255:red, 0; green, 0; blue, 0 }  ,draw opacity=1 ] [dash pattern={on 4.5pt off 4.5pt}]  (373.14,261.54) -- (389.8,286.68) ;
\draw [color={rgb, 255:red, 0; green, 0; blue, 0 }  ,draw opacity=1 ] [dash pattern={on 4.5pt off 4.5pt}]  (389.8,286.68) -- (388,255.48) ;
\draw [color={rgb, 255:red, 0; green, 0; blue, 0 }  ,draw opacity=1 ] [dash pattern={on 4.5pt off 4.5pt}]  (390.97,318.45) -- (390.31,288.29) ;
\draw [color={rgb, 255:red, 0; green, 0; blue, 0 }  ,draw opacity=1 ] [dash pattern={on 4.5pt off 4.5pt}]  (390.31,288.29) -- (406.64,314.94) ;
\draw  [color={rgb, 255:red, 0; green, 0; blue, 0 }  ,draw opacity=1 ] (386,288.77) .. controls (385.38,286.51) and (386.7,284.17) .. (388.96,283.54) .. controls (391.22,282.91) and (393.57,284.24) .. (394.19,286.5) .. controls (394.82,288.76) and (393.5,291.1) .. (391.24,291.73) .. controls (388.98,292.36) and (386.63,291.03) .. (386,288.77) -- cycle ;

\draw  [draw opacity=0][dash pattern={on 0.84pt off 2.51pt}] (392.54,320.11) .. controls (396.06,319.88) and (399.6,319.03) .. (403.02,317.52) .. controls (403.95,317.11) and (404.86,316.66) .. (405.73,316.17) -- (390.31,288.29) -- cycle ; \draw  [color={rgb, 255:red, 0; green, 0; blue, 0 }  ,draw opacity=1 ][dash pattern={on 0.84pt off 2.51pt}] (392.54,320.11) .. controls (396.06,319.88) and (399.6,319.03) .. (403.02,317.52) .. controls (403.95,317.11) and (404.86,316.66) .. (405.73,316.17) ;

\draw  [draw opacity=0][dash pattern={on 0.84pt off 2.51pt}] (373.5,259.27) .. controls (376.52,257.45) and (379.91,256.13) .. (383.58,255.42) .. controls (384.56,255.23) and (385.53,255.09) .. (386.51,255) -- (389.8,286.68) -- cycle ; \draw  [color={rgb, 255:red, 0; green, 0; blue, 0 }  ,draw opacity=1 ][dash pattern={on 0.84pt off 2.51pt}] (373.5,259.27) .. controls (376.52,257.45) and (379.91,256.13) .. (383.58,255.42) .. controls (384.56,255.23) and (385.53,255.09) .. (386.51,255) ;

\draw [color={rgb, 255:red, 0; green, 0; blue, 0 }  ,draw opacity=1 ] [dash pattern={on 4.5pt off 4.5pt}]  (349.13,258.32) -- (344.08,288.06) ;
\draw [color={rgb, 255:red, 0; green, 0; blue, 0 }  ,draw opacity=1 ] [dash pattern={on 4.5pt off 4.5pt}]  (344.08,288.06) -- (364.12,264.08) ;
\draw [color={rgb, 255:red, 0; green, 0; blue, 0 }  ,draw opacity=1 ] [dash pattern={on 4.5pt off 4.5pt}]  (323.2,312.03) -- (343.35,289.58) ;
\draw [color={rgb, 255:red, 0; green, 0; blue, 0 }  ,draw opacity=1 ] [dash pattern={on 4.5pt off 4.5pt}]  (343.35,289.58) -- (337.03,320.19) ;
\draw  [color={rgb, 255:red, 0; green, 0; blue, 0 }  ,draw opacity=1 ] (339.88,286.99) .. controls (340.97,284.91) and (343.54,284.11) .. (345.62,285.19) .. controls (347.7,286.28) and (348.5,288.85) .. (347.41,290.93) .. controls (346.32,293.01) and (343.75,293.81) .. (341.67,292.72) .. controls (339.59,291.63) and (338.79,289.07) .. (339.88,286.99) -- cycle ;

\draw  [draw opacity=0][dash pattern={on 0.84pt off 2.51pt}] (323.21,314.31) .. controls (325.94,316.55) and (329.1,318.36) .. (332.62,319.6) .. controls (333.58,319.94) and (334.55,320.22) .. (335.52,320.46) -- (343.35,289.58) -- cycle ; \draw  [color={rgb, 255:red, 0; green, 0; blue, 0 }  ,draw opacity=1 ][dash pattern={on 0.84pt off 2.51pt}] (323.21,314.31) .. controls (325.94,316.55) and (329.1,318.36) .. (332.62,319.6) .. controls (333.58,319.94) and (334.55,320.22) .. (335.52,320.46) ;

\draw  [draw opacity=0][dash pattern={on 0.84pt off 2.51pt}] (350.94,256.91) .. controls (354.4,257.66) and (357.77,259.01) .. (360.93,261) .. controls (361.78,261.53) and (362.58,262.1) .. (363.36,262.7) -- (344.08,288.06) -- cycle ; \draw  [color={rgb, 255:red, 0; green, 0; blue, 0 }  ,draw opacity=1 ][dash pattern={on 0.84pt off 2.51pt}] (350.94,256.91) .. controls (354.4,257.66) and (357.77,259.01) .. (360.93,261) .. controls (361.78,261.53) and (362.58,262.1) .. (363.36,262.7) ;

\draw [color={rgb, 255:red, 0; green, 0; blue, 0 }  ,draw opacity=1 ] [dash pattern={on 4.5pt off 4.5pt}]  (335.52,234.58) -- (312.48,254.04) ;
\draw [color={rgb, 255:red, 0; green, 0; blue, 0 }  ,draw opacity=1 ] [dash pattern={on 4.5pt off 4.5pt}]  (312.48,254.04) -- (343.26,248.65) ;
\draw [color={rgb, 255:red, 0; green, 0; blue, 0 }  ,draw opacity=1 ] [dash pattern={on 4.5pt off 4.5pt}]  (281.06,258.88) -- (310.94,254.73) ;
\draw [color={rgb, 255:red, 0; green, 0; blue, 0 }  ,draw opacity=1 ] [dash pattern={on 4.5pt off 4.5pt}]  (310.94,254.73) -- (286.36,274.04) ;
\draw  [color={rgb, 255:red, 0; green, 0; blue, 0 }  ,draw opacity=1 ] (309.96,250.52) .. controls (312.14,249.63) and (314.62,250.67) .. (315.5,252.85) .. controls (316.39,255.02) and (315.34,257.5) .. (313.17,258.39) .. controls (311,259.27) and (308.52,258.23) .. (307.63,256.05) .. controls (306.75,253.88) and (307.79,251.4) .. (309.96,250.52) -- cycle ;

\draw  [draw opacity=0][dash pattern={on 0.84pt off 2.51pt}] (279.59,260.64) .. controls (280.23,264.11) and (281.48,267.53) .. (283.38,270.75) .. controls (283.89,271.63) and (284.45,272.47) .. (285.04,273.28) -- (310.94,254.73) -- cycle ; \draw  [color={rgb, 255:red, 0; green, 0; blue, 0 }  ,draw opacity=1 ][dash pattern={on 0.84pt off 2.51pt}] (279.59,260.64) .. controls (280.23,264.11) and (281.48,267.53) .. (283.38,270.75) .. controls (283.89,271.63) and (284.45,272.47) .. (285.04,273.28) ;

\draw  [draw opacity=0][dash pattern={on 0.84pt off 2.51pt}] (337.82,234.68) .. controls (339.98,237.47) and (341.68,240.69) .. (342.81,244.25) .. controls (343.11,245.19) and (343.37,246.15) .. (343.57,247.1) -- (312.48,254.04) -- cycle ; \draw  [color={rgb, 255:red, 0; green, 0; blue, 0 }  ,draw opacity=1 ][dash pattern={on 0.84pt off 2.51pt}] (337.82,234.68) .. controls (339.98,237.47) and (341.68,240.69) .. (342.81,244.25) .. controls (343.11,245.19) and (343.37,246.15) .. (343.57,247.1) ;

\draw [color={rgb, 255:red, 0; green, 0; blue, 0 }  ,draw opacity=1 ] [dash pattern={on 4.5pt off 4.5pt}]  (282.15,203.27) -- (311.96,207.88) ;
\draw [color={rgb, 255:red, 0; green, 0; blue, 0 }  ,draw opacity=1 ] [dash pattern={on 4.5pt off 4.5pt}]  (311.96,207.88) -- (287.68,188.2) ;
\draw [color={rgb, 255:red, 0; green, 0; blue, 0 }  ,draw opacity=1 ] [dash pattern={on 4.5pt off 4.5pt}]  (336.24,228.4) -- (313.49,208.59) ;
\draw [color={rgb, 255:red, 0; green, 0; blue, 0 }  ,draw opacity=1 ] [dash pattern={on 4.5pt off 4.5pt}]  (313.49,208.59) -- (344.19,214.45) ;
\draw  [color={rgb, 255:red, 0; green, 0; blue, 0 }  ,draw opacity=1 ] (310.95,212.09) .. controls (308.85,211.03) and (308.01,208.48) .. (309.07,206.38) .. controls (310.13,204.29) and (312.69,203.45) .. (314.78,204.51) .. controls (316.88,205.56) and (317.72,208.12) .. (316.66,210.21) .. controls (315.6,212.31) and (313.05,213.15) .. (310.95,212.09) -- cycle ;

\draw  [draw opacity=0][dash pattern={on 0.84pt off 2.51pt}] (338.52,228.36) .. controls (340.72,225.6) and (342.48,222.41) .. (343.67,218.87) .. controls (343.99,217.9) and (344.26,216.93) .. (344.49,215.95) -- (313.49,208.59) -- cycle ; \draw  [color={rgb, 255:red, 0; green, 0; blue, 0 }  ,draw opacity=1 ][dash pattern={on 0.84pt off 2.51pt}] (338.52,228.36) .. controls (340.72,225.6) and (342.48,222.41) .. (343.67,218.87) .. controls (343.99,217.9) and (344.26,216.93) .. (344.49,215.95) ;

\draw  [draw opacity=0][dash pattern={on 0.84pt off 2.51pt}] (280.71,201.48) .. controls (281.41,198.01) and (282.71,194.62) .. (284.66,191.43) .. controls (285.18,190.58) and (285.73,189.76) .. (286.31,188.98) -- (311.96,207.88) -- cycle ; \draw  [color={rgb, 255:red, 0; green, 0; blue, 0 }  ,draw opacity=1 ][dash pattern={on 0.84pt off 2.51pt}] (280.71,201.48) .. controls (281.41,198.01) and (282.71,194.62) .. (284.66,191.43) .. controls (285.18,190.58) and (285.73,189.76) .. (286.31,188.98) ;

\end{tikzpicture}
\begin{adjustwidth}{65pt}{65pt}
       \caption
       {\small An illustration for collection of  punctured blind cones $\Gamma_{x_0}(c,v) \subset {\rm I\!R}^n \times {\rm I\!R}^n$. In this drawing balls are of radius $v$ and cones with apex angle $c$ are facing the observer $x_0$ at the center.}
\end{adjustwidth}
\end{figure}
\noindent Using the boundedness of the relative angular norm, as a consequence of Theorem \ref{RANB}, we will control the total amount of mass inside $N_{R}(x_0)$. Start with the following observation:
\begin{align*}
    \iint\limits_{N_{R}(x_0)} f(x,\xi,t)(|x||\xi|-x\cdot\xi) \ dxd\xi \leq \|f\|_{G}
\end{align*}
Let $\theta(x,\xi)$ be the angle between $x$ and $\xi$: 
\begin{align*}
   \iint\limits_{N_{R}(x_0)} f(x,\xi,t)(|x||\xi|-x\cdot\xi) \ dxd\xi = \iint\limits_{N_{R}(x_0)} f(x,\xi,t)|x||\xi|(1-\cos(\theta(x,\xi))) \ dxd\xi
 \end{align*}
Therefore:
\begin{align*}
  Rv(1-\cos(c)) \iint\limits_{N_{R}(x_0)} f(x,\xi,t)\ dxd\xi< \iint\limits_{N_{R}(x_0)} f(x,\xi,t)(|x||\xi| - x\cdot\xi)\ dxd\xi\leq \|f\|_{G}
   \end{align*}
And finally we get:
\begin{align}
  \iint\limits_{N_{R}(x_0)} f(x,\xi,t)\ dxd\xi < \frac{\|f\|_{G}}{Rv(1-\cos(c)) } 
\end{align} 
\noindent Now continue with the conservation of mass over time: 
\begin{align*}
\lim_{T\rightarrow\infty} \frac{1}{T}\int_{0}^{T}\int_{{\rm I\!R}^n}\int_{{\rm I\!R}^n} f(x,\xi, t)\ dxd\xi dt = {\bf M}
\end{align*}
We will break the space of velocities and positions in to the three non intersecting subsets as below: 
\begin{multline*}
\lim_{T\rightarrow\infty} \frac{1}{T}\int_{0}^{T}\Big(\iint\limits_{N_{R}(x_0)} f(x,\xi, t)\ dxd\xi+\iint\limits_{M_{R}(x_0)} f(x,\xi, t)\ dxd\xi\\+\iint\limits_{\Gamma_{x_{0}(c,v)}} f(x,\xi, t)\ dxd\xi \Big)dt ={\bf M} 
\end{multline*}
As a consequence of Corollary \ref{dispersionM} the total amount of mass within $M_{R}(x_0)$ is integrable over time, therefore the equation above leads to:  
\begin{align*}
\lim_{T\rightarrow\infty} \frac{1}{T}\int_{0}^{T}\big(\iint\limits_{N_{R}(x_0)} f(x,\xi, t)\ dxd\xi+\iint\limits_{\Gamma_{x_{0}(c,v)}} f(x,\xi, t)\ dxd\xi \big)dt={\bf M}
\end{align*}
Using (2.6) as a bound for the total amount of mass in $N_{R}(x_0)$ we get: 
\begin{align*}
{\bf M} - \frac{\|f\|_{G}}{Rv(1-\cos(c))}  \leq \liminf_{T\rightarrow\infty} \frac{1}{T}\int_{0}^{T}(\iint\limits_{\Gamma_{x_{0}(c,v)}} f(x,\xi, t)\ dxd\xi)dt \leq {\bf M}
\end{align*}
Because $R$ can be arbitrary large and $\|f\|_{G}$ is bounded, finally we conclude: 
\begin{align*}
    \lim_{T\rightarrow\infty} \frac{1}{T}\int_{0}^{T}(\iint\limits_{\Gamma_{x_0}(c,v)} f(x,\xi, t)\ dxd\xi)dt = {\bf M}
\end{align*}
\end{proof}
\begin{remark} The previous theorem implies that the total amount of mass concentrates in $\Gamma_{x_0}(c,v)$ as time goes to infinity. Equivalently one could say that on average over time, the mass is contained almost always within the collection of punctuate blind cones. Consider $\Gamma_{x_0}(c,v) \subset {\rm I\!R}^n \times {\rm I\!R}^n$ for any $x_0 \in {\rm I\!R}^n$ and arbitrarily small $c$ and $v$.
\end{remark}

We will finish this section by pointing out that combining Corollary \ref{dispersionM} with the previous theorem can improve its result. To do so we replace the punctured blind cones in the definition of $\Gamma_{x_0}(c,v)$ with the balls of radius $v$ within any fixed bounded region like $D$ and keep  $\Gamma_{x_0}(c,v)$ intact outside of the bounded set $D$. 
\newpage

\section{Scattering phenomenon and existence theory}
A possible and beneficial narrative that one can attribute to the results of the previous section, that are proven independent from the specific structure of interactions, is that they are evidence for an inevitable weak dispersion of particles. By weak we mean averaged over time or within a bounded set of the spatial variable. We will use the term scattering as a specific type of dispersion, that is a strong dispersion in a norm. This subject will naturally bring a discussion of existence theory with itself. The ideas discussed here will be used in Section 4 for the  purpose of creating a specific class of solutions to the Boltzmann equation.\\

\indent One could argue to start with the bare minimum, by searching for a norm over the space of positions and velocities, which if for any specific slice of time it is bounded, then the energy, mass, momentum and interactions are well defined at that time. 

\begin{definition}\label{indicator} Let $\chi_{D}(x)$ be the indicator function for set $D \subset {\rm I\!R}^n$: 
\begin{align*}
    \chi_{D}(x)=1 \ \ \ \ x\in D\\ 
      \chi_{D}(x)=0 \ \ \ \ x\notin D
\end{align*}
\end{definition} 
\begin{definition}\label{onorm}
We say $\|\|_{O}$ is an {\it observer's norm} if it is defined on the spatial and velocity variables for a fixed time. Furthermore, we expect that if observer's norm of $f$ is bounded at time $t$, then interaction $I(f, x, \xi, t)$ is well defined and mass, momentum and energy are convergent integrals at that slice of time:
\begin{align*}
     &\int_{{\rm I\!R}^n}\int_{{\rm I\!R}^n} f(x,\xi,t)\ dxd\xi < \infty \\ \|f\|_{O}(t) < \infty   \Longrightarrow &\int_{{\rm I\!R}^n}\int_{{\rm I\!R}^n} f(x,\xi,t)\xi \ dxd\xi < \infty \\ & \int_{{\rm I\!R}^n}\int_{{\rm I\!R}^n} f(x,\xi,t)|\xi|^2 \ dxd\xi < \infty \\
 \end{align*}
\end{definition}
\begin{remark}
Note that the observers norm is acting only on $x$ and $\xi$ variables for a fixed $t$, therefore its value could possibly be a function of time. This norm represents the measurement of physical descriptions such as position and velocity with respect to an observer at a slice of time. 
\end{remark}
\begin{definition}\label{snorm}
Let $\|\|_{S}$ be the {\it scattering norm} associated to the observers norm $\|\|_{O}$ defined as: 
\begin{align*}
    \|f\|_{S} = \sup_{t}\| f(x + t\xi, \xi, t)\|_{O}
\end{align*}
\end{definition}
\begin{remark}
The scattering norm involves all of the three variables, $x$, $\xi$ and $t$. Therefore, if it is bounded it will be constant.  Consider that if $\lambda(x,\xi,t)$ is a solution to the linear transport equation, then its scattering norm is equal to its observer's norm at time zero: $\|\lambda(x,\xi,t)\|_{S}$=$\|\lambda(x,\xi,0)\|_{O}$. The scattering norm is built upon an observer's norm by assigning an observer to each characteristic and following the particles along their trajectories.
\end{remark}
\begin{definition}\label{sospaces} For $0<C$, let $S(C)$ be the {\it scattering space} associated to scattering norm $\|\|_{S}$ defined as: 
\begin{align*}
     f(x, \xi, t) \in S(C) \Longleftrightarrow\|f\|_{S} \leq C
\end{align*}
\end{definition}
\begin{remark}
Consider that the scattering space defined above is not a vector spaces and is solely a space with a norm. It is possible to think of this space as ball of finite radius within a larger infinite dimensional vector space. We will identify the scattering space with its completion. 
\end{remark}

\begin{definition}\label{scattering solution}
$f$ is a {\it scattering solution} of equation (1.1) relative to an observer's norm $\|\|_{O}$ if: \begin{align*}
\|f(x, \xi, 0)\|_{O} +
\| \int_{0}^{\infty} |I(f,x+t\xi, \xi, t)| dt\|_{O} < \infty
\end{align*}
And for any bounded subset of the spatial variable like $D \subset {\rm I\!R}^n$ we expect: 
\begin{align}
    \lim_{z\rightarrow \infty }\|\chi_{D}(x) \int_{z}^{\infty} |I(f,x+t\xi, \xi, t)| dt\|_{O} = 0 
\end{align}
\noindent We say $f$ is a {\it uniform scattering solution} if the condition below, which is stronger than (3.1), holds true:
\begin{align}
    \lim_{z\rightarrow \infty }\| \int_{z}^{\infty} |I(f,x+t\xi, \xi, t)| dt\|_{O} = 0 
\end{align}
\end{definition}

\begin{remark} It is possible to represent a scattering solution like $f$ by integrating over the characteristics of equation (1.1): 
 \begin{align*}
    f(x,\xi, t) = f_{0}(x - t\xi, \xi) + \int_{0}^{t}I(f,x - t\xi+ s\xi,\xi, s)ds 
\end{align*}
The series of identities below describe the relation between the scattering and observer's norm of this solution. The following transition between the two norms occurs frequently hereinafter. 
\begin{align*}
\|f(x,\xi,t)\|_{S} = \sup_{t}\|f(x+t\xi,\xi,t)\|_{O}=\sup_{t}\|f_{0}(x, \xi) + \int_{0}^{t}I(f,x + s\xi,\xi, s)ds\|_{O}
\end{align*}
Therefore, the following bound holds true for a scattering norm: 
\begin{align*}
\|f(x,\xi,t)\|_{S} \leq \|f(x, \xi, 0)\|_{O} +\|\int_{0}^{\infty} |I(x+t\xi, \xi,t)| dt\|_{O} 
\end{align*}
Equivalently we have: 
\begin{align*}
     f \in S\big(\|f(x, \xi, 0)\|_{O}+  \|\int_{0}^{\infty} |I(x+t\xi, \xi,t)| dt\|_{O}\big)
 \end{align*}
\end{remark}
\begin{figure}[t]
    \centering
\tikzset{every picture/.style={line width=0.75pt}} 

\begin{tikzpicture}[x=0.75pt,y=0.75pt,yscale=-1,xscale=1]

\draw    (133.5,165) .. controls (158,49.5) and (183,51.5) .. (196,91.5) .. controls (209,131.5) and (230,263.5) .. (237.5,200) .. controls (245,136.5) and (267.5,230) .. (271.5,171) .. controls (275.5,112) and (283.5,168) .. (289.5,166) .. controls (295.5,164) and (294.5,138) .. (308.5,152) ;
\draw    (132.5,18.5) -- (132.5,236) ;
\draw  [dash pattern={on 4.5pt off 4.5pt}]  (311.5,154) .. controls (326.5,164) and (328.5,152) .. (334.5,144) .. controls (340.5,136) and (348.5,148) .. (378.5,142) ;
\draw    (133.5,211.75) -- (340.5,158) ;
\draw  [dash pattern={on 4.5pt off 4.5pt}]  (340.5,158) -- (354.25,154.91) -- (380.5,149) ;
\draw    (132,124.25) -- (340.5,71) ;
\draw  [dash pattern={on 4.5pt off 4.5pt}]  (340.5,71) -- (380.5,62) ;
\draw [line width=0.75]    (133,236) -- (385,236) ;
\draw [shift={(388,236)}, rotate = 180] [fill={rgb, 255:red, 0; green, 0; blue, 0 }  ][line width=0.08]  [draw opacity=0] (8.93,-4.29) -- (0,0) -- (8.93,4.29) -- cycle    ;
\draw  [dash pattern={on 0.84pt off 2.51pt}]  (199.69,112.31) -- (199.69,233.38) ;
\draw  [fill={rgb, 255:red, 255; green, 255; blue, 255 }  ,fill opacity=1 ] (196.88,107.5) .. controls (196.88,105.95) and (198.13,104.69) .. (199.69,104.69) .. controls (201.24,104.69) and (202.5,105.95) .. (202.5,107.5) .. controls (202.5,109.05) and (201.24,110.31) .. (199.69,110.31) .. controls (198.13,110.31) and (196.88,109.05) .. (196.88,107.5) -- cycle ;
\draw  [fill={rgb, 255:red, 255; green, 255; blue, 255 }  ,fill opacity=1 ] (129.9,124.44) .. controls (129.9,122.88) and (131.16,121.63) .. (132.71,121.63) .. controls (134.27,121.63) and (135.53,122.88) .. (135.53,124.44) .. controls (135.53,125.99) and (134.27,127.25) .. (132.71,127.25) .. controls (131.16,127.25) and (129.9,125.99) .. (129.9,124.44) -- cycle ;
\draw  [fill={rgb, 255:red, 255; green, 255; blue, 255 }  ,fill opacity=1 ] (129.9,211.75) .. controls (129.9,210.2) and (131.16,208.94) .. (132.71,208.94) .. controls (134.27,208.94) and (135.53,210.2) .. (135.53,211.75) .. controls (135.53,213.3) and (134.27,214.56) .. (132.71,214.56) .. controls (131.16,214.56) and (129.9,213.3) .. (129.9,211.75) -- cycle ;
\draw  [fill={rgb, 255:red, 255; green, 255; blue, 255 }  ,fill opacity=1 ] (196.88,235.19) .. controls (196.88,233.63) and (198.13,232.38) .. (199.69,232.38) .. controls (201.24,232.38) and (202.5,233.63) .. (202.5,235.19) .. controls (202.5,236.74) and (201.24,238) .. (199.69,238) .. controls (198.13,238) and (196.88,236.74) .. (196.88,235.19) -- cycle ;
\draw  [fill={rgb, 255:red, 255; green, 255; blue, 255 }  ,fill opacity=1 ] (129.88,167) .. controls (129.88,165.45) and (131.13,164.19) .. (132.69,164.19) .. controls (134.24,164.19) and (135.5,165.45) .. (135.5,167) .. controls (135.5,168.55) and (134.24,169.81) .. (132.69,169.81) .. controls (131.13,169.81) and (129.88,168.55) .. (129.88,167) -- cycle ;

\draw (105.5,110) node [anchor=north west][inner sep=0.75pt]    {$f_{z}$};
\draw (100.5,201) node [anchor=north west][inner sep=0.75pt]    {$f_{\infty }$};
\draw (200,238) node [anchor=north west][inner sep=0.75pt]    {$z$};
\draw (188,41.5) node [anchor=north west][inner sep=0.75pt]    {$f$};

\end{tikzpicture}

\begin{adjustwidth}{65pt}{65pt}
           \caption{\small $f_z$ and $f_{\infty}$ are solutions to the linear transport equation. As time goes to infinity, $f$ scatters to $f_{\infty}$ in a specific sense.} 
\end{adjustwidth}
\end{figure}
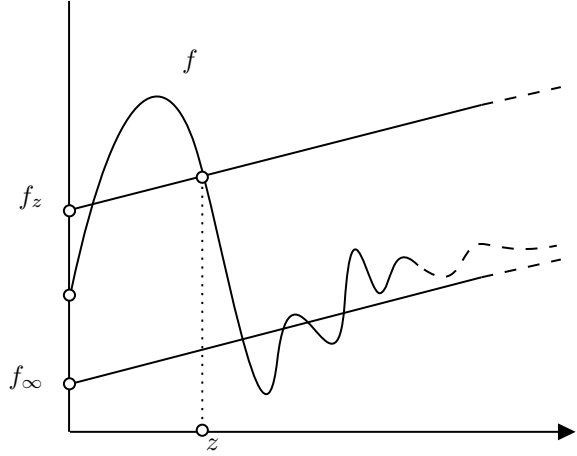
\begin{theorem}\label{scatteringthm} Assume that $f$ is a scattering solution of equation (1.1):  
\begin{align} N = \|f(x, \xi, 0)\|_{O}+ \|\int_{0}^{\infty} |I(x+t\xi, \xi,t)| dt\|_{O} < \infty
\end{align} 
There exists a unique solution of the linear transport equation $f_{\infty}\in S(N)$, such that $f$ {\it scatters} to $f_{\infty}$ within any compact subset of the spatial variable like $D \subset {\rm I\!R}^n$ in the sense defined below:
\begin{align*}
    \lim_{t \rightarrow \infty}\|\chi_D(x)\big(f(x+t\xi,\xi,t) - f_{\infty}(x,\xi, 0)\big)\|_{O}=0
\end{align*}

\end{theorem}
\begin{proof}
 Consider Figure 4 and let $f_{z}(x,\xi,t)$ for $0\leq z$ be the linearization of $f$ at time $z$ defined as: 
\begin{align*}
    f_{z}(x, \xi, t)=f(x-(t-z)\xi,\xi,z)
\end{align*}
Using the definition above we will construct a family of solutions to the linear transport equation by cropping $f_z$ within an arbitrary compact subset of the spatial variable like $D \subset {\rm I\!R}^n$:    
\begin{align*}
    \chi_D(x - t\xi)f_z(x,\xi,t)
\end{align*}

\noindent Note that $f_{z}$ and $\chi_D(x - t\xi)f_z(x,\xi,t)$ belong to $S(N)$. Using the properties of scattering solutions, we will show that $\chi_D(x - t\xi)f_z(x,\xi,t)$ is a Cauchy sequence. Recall definitions of the observer's and scattering norms, it follows that for $0\leq z_1 \leq z_2$ we get: 
\begin{multline*}
   \|\chi_D(x - t\xi)\big(f(x-(t-z_{1})\xi,\xi,z_{1})-f(x-(t-z_{2})\xi,\xi,z_{2})\big)\|_{S}\\ = \|\chi_D(x - t\xi)\int_{z_{1}}^{z_{2}}I(f, x-(t-z)\xi, \xi, z)dz\|_{S} \\ = \sup_{t} \|\chi_D(x)\int_{z_{1}}^{z_{2}}I(f, x+z\xi, \xi, z)dz\|_{O}
\end{multline*}
The previous computation leads to:
\begin{align}
     \| \chi_D(x - t\xi)(f_{z_1}(x,\xi,t) - f_{z_{2}}(x,\xi,t))\|_{S} = \|\chi_D(x )\int_{z_1}^{z_2} |I(f,x+z\xi,\xi, z)|dz\|_{O}
\end{align}
\noindent Since $f$ is a scattering solution it follows from property (3.1) that:    
\begin{align*}
    \forall \epsilon > 0 \  \exists T > 0 \ \ s.t \ \ \ T< z_1\leq z_2 \Longrightarrow   \|\chi_D(x )\int_{z_1}^{z_2} |I(f,x+z\xi,\xi, z)|dz\|_{O}  < \epsilon
\end{align*} 
\noindent The previous statement paired with (3.4)  implies that the sequence of functions $\chi_D(x-t\xi)f_{z}(x,\xi,t)$ is a Cauchy sequence converging to another solution of the linear transport like $f^{D}_{\infty} \in S(N)$. We conclude that:
\begin{equation}
\begin{aligned}
     \lim_{z\rightarrow\infty}\|\chi_{D}(x-t\xi)(f_z(x,\xi,t) - f^{D}_{\infty}(x,\xi,t))\|_{S}=0 
\end{aligned}
\end{equation}
The family of linear states $f^{D}_{\infty}$ are corresponding to different compact subsets of the spatial variable. We will show it is possible to consolidate these functions in a consistent manner and create a unique solution to the transport equation that satisfies the claims of the theorem. Consider the identities below: 
\begin{multline*}
    \|\chi_{D}(x-t\xi)(f_z(x,\xi,t) - f^{D}_{\infty}(x,\xi,t))\|_{S}\\ = \sup_{t}\|\chi_{D}(x)(f_z(x+t\xi,\xi,t) - f^{D}_{\infty}(x+t\xi,\xi,t))\|_{O}\\= \|\chi_{D}(x)(f(x+z\xi, \xi,z)-f^{D}_{\infty}(x,\xi,0))\|_{O}  
\end{multline*}
Therefore from (3.5) we get: 
\begin{align*}
           \lim_{z\rightarrow\infty}\|\chi_{D}(x)(f(x+z\xi, \xi,z)-f^{D}_{\infty}(x,\xi,0))\|_{O}=0
\end{align*}
Assume that $D_1,D_2 \subset {\rm I\!R}^n$ are two arbitrary compact sets. From the previous computation we conclude that the values of $f^{D_1}_{\infty}$ and $f^{D_2}_{\infty}$ agree at time zero wherever $D_1$ and $D_2$ overlap:   
\begin{equation}
\begin{aligned}
    x \in D_{1} \cap D_{2} \Longrightarrow f^{D_1}_{\infty}(x,\xi,0) = f^{D_2}_{\infty}(x,\xi,0)=\lim_{z\rightarrow\infty}f(x+z\xi, \xi, z)
\end{aligned}
\end{equation}
Now let $f_{\infty}(x,\xi,t) \in S(N)$ be a solution the linear transport equation evolving from the initial value defined below, where $B(x,1) \subset {\rm I\!R}^n$ is a ball of radius 1 centered at $x$.  
\begin{align*}
    f_{\infty}(x,\xi,0)=f^{B(x,1)}_{\infty}(x,\xi,0)
\end{align*}
 We deduce from (3.6) that $f_{\infty}$ is well defined at time zero by the expression above. Note that the ball centered at $x$ can be replaced by any arbitrary open neighborhood around $x$. Finally using (3.5) again we get: 
\begin{align*}
    \lim_{t \rightarrow \infty}\|\chi_D(x)\big(f(x+t\xi,\xi,t) - f_{\infty}(x,\xi, 0)\big)\|_{O}=0
\end{align*}

\end{proof}

\begin{corollary}\label{uscatteringc}
Assume $f$ is a uniform scattering solution to the equation (1.1) and $N$ is as defined in (3.3). There exists a unique solution to the linear transport equation $f_{\infty} \in S(N)$, such that $f$ {\it uniformly scatters} to $f_{\infty}$ in the sense defined below:
\begin{align*}
    \lim_{t\rightarrow\infty}\|f(x+t\xi,\xi,t) - f_{\infty}(x,\xi,0)\|_{O} = 0
\end{align*}
\begin{proof}
Under the stronger assumption of having a uniform scattering solution, the argument of Theorem \ref{scatteringthm} can be replicated without any reference to the bounded sets. In that case, by removing $\chi_{D}(x)$ from the computations before, one can prove that the sequence $f_{z}$ is a Cauchy sequence converging to a unique solution of the transport equation $f_{\infty} \in S(N)$.
\end{proof}
\end{corollary}
\indent We will end this section with a discussion of the existence theory for scattering solutions to equation (1.1). Theorem \ref{scatteringthm} and Corollary \ref{uscatteringc} provide a natural procedure to construct scattering solutions. Assume that $f$ is a scattering solution that scatters to a linear state like $f_{\infty}$. Choose $\lambda(x,\xi,t)$ to be any positive solution of the linear transport equation such that: 
 \begin{align*} f_{\infty}(x+t\xi,\xi,t)=f_{\infty}(x,\xi,0)\leq \lambda(x,\xi,0)=\lambda(x+t\xi,\xi,t)
\end{align*}
\noindent This positive $\lambda$ can be used for creation of observer's and scattering norms associated to $f$. In particular, the norms below are well defined: 
\begin{align*}
   \|f\|_{O} &= \sup_{x,\xi}\frac{1}{\lambda(x+t\xi,\xi,t)} f(x,\xi, t) < \infty \\ \|f\|_{S} &=\sup_{x,\xi,t}\frac{1}{\lambda(x+t\xi,\xi,t)} f(x+t\xi,\xi, t) < \infty 
\end{align*}
Working from the converse direction with an appropriate $\lambda$, one would hope to create a class of solutions that scatter to linear states below $\lambda$. Consider the solution map for equation (1.1) defined below. The fixed points of this map are solutions of the equation: 
\begin{align*}
    \Phi(f)(x,\xi, t) = f_{0}(x - t\xi, \xi) + \int_{0}^{t}I(f,x - t\xi+ s\xi,\xi, s)ds   
\end{align*}
\indent Assume that the initial value $f_{0}$ is sufficiently below $\lambda$ and that appropriate estimates for interaction $I(\lambda, x, \xi,t)$ along the characteristic of equation (1.1) exist. Using these estimates one would hope to prove that solutions starting from such initial values will remain uniformly bounded by $\lambda$. In this setting we will construct a scattering space such that iterations of the solution map remain within the same space. Therefore, a variety of fixed point methods become tools to develop an existence theory. The question whether this map has a fixed point requires more information on the specific structure of interactions.\\
 
\indent The notion of scattering that got developed in this section led us to the previous discussion on the existence theory of solutions which converge to a linear state. This framework is applicable to a variety of dispersive non linear partial differential equations. We will implement these ideas in the next chapter in the case of the Boltzmann equation and show the existence of a class of scattering solutions that enjoy asymptotic completeness and stability in the $L^{\infty}$ setting.

\newpage 
\section{The Boltzmann equation}
In this section, we briefly introduce the Boltzmann equation in the case of hard spheres. A comprehensive treatment of the mathematical theory behind the classical solutions to the Boltzmann equation can be found in the works of Bressan \cite{bressan}, Cercignani \cite {MR1307620, MR1313028}, and Ukai \cite{senji}. Furthermore, a concept of renormalized solutions exists as introduced by DiPerna and Lions \cite{MR1014927, MR1022305} that provides a very general existence theory of weak solutions. The counterpart of this renormalization in the microscopic scale is unknown. This missing link makes it infeasible to completely relate these weak solutions to the underlying physical phenomena which the Boltzmann equation is intended to describe, since  the conservation laws are rooted in the dynamics of microscopic scale. \\

\indent Here, we pursue classical solutions to demonstrate the utility of the previous sections' results. After introducing the preliminaries we will show the existence of scattering solutions to the Boltzmann equation in the case of hard spheres, for a class of initial values with a special property. Furthermore, we will show that for any solution to the linear transport equation in an appropriate space, there exists solutions to the Boltzmann equation that scatter to linear states arbitrarily close to the previously described solution of the linear transport equation. 

\subsection{Preliminaries}
The Boltzmann equation has a similar form to equation (1.1). The notation $Q(f,f)(x,\xi,t)$ represents the interaction of particles and is called the Boltzmann collision operator:  
\begin{align*}
\partial_t f(x,\xi,t) + \xi.\nabla_x f(x,\xi,t) &= Q(f, f)(x,\xi,t) \\  f(0, x, \xi) &= f_0(x,\xi) 
\end{align*}
In the case which this operator represents collisions of hard elastic spheres it obtains the form below:
\begin{gather*}
    Q(f, f)(x,\xi,t) = \int_{{\rm I\!R}^n}\int_{S^{n-1}}(f^{'} f^{'}_{*} - f f_{*})|{\bf n}.(\xi- \xi_{*})| d{\bf n}d\xi_*  
 \end{gather*}
In the expression above, ${\bf n}$ is the normal unit vector to  the $n-1$ dimensional unit sphere $S^{n-1}$ and, as is customary in the field, we used the notation below to describe the velocity of particles before and after the collisions: 
\begin{gather*}
    \xi^{'} = \xi - {\bf n}.(\xi- \xi_{*}){\bf n}  \\
     \xi^{'}_{*} = \xi_{*} + {\bf n}.(\xi- \xi_{*}){\bf n}  \\ 
f = f(x,\xi,t),\ f_{*} = f(x,\xi_*,t) \\ f^{'} = f(x, \xi^{'}, t) , \ f^{'}_{*} = f(x, \xi^{'}_{*}, t) 
\end{gather*}
\indent Assume that $\xi$ and $\xi_{*}$ are velocities of two colliding elastic balls with unit mass, and let ${\bf n}$ or equivalently $-{\bf n}$, represent the unit normal vector to the plane which uniquely describes the relative position of these two spheres upon the collision. Then $\xi^{'}$ and $\xi_{*}^{'}$ are the velocities of the  particles after the collision. For a fixed ${\bf n}$, these velocities are unique solutions to the conservation laws of momentum and energy written below: 
\begin{align*}
    |\xi|^2 + |\xi_{*}|^2 &= |\xi^{'}|^2 + |\xi^{'}_{*}|^2 \\ \xi + \xi_{*} &= \xi^{'}+\xi_{*}^{'}
\end{align*}

As shown by Boltzmann, the collision operator satisfies the conditions of Definition \ref{interaction} and therefore is a mesoscopic interaction. We will replicate his argument here. Consider the  well-known change of variables below:
\begin{align*}
    (\xi, \xi_{*}) \Longrightarrow (\xi_* , \xi), \ \ \ (\xi, \xi_{*}) \Longrightarrow (\xi^{'} , \xi_{*}^{'}), \ \ \   (\xi, \xi_{*}) \Longrightarrow (\xi^{'}_* , \xi^{'})  
\end{align*}
The transformations above are measure persevering. They represent the intrinsic symmetries of the conservation laws, in which role of the particles as well as their velocities before and after a collision are indistinguishable. We will implement these change of variables for the Boltzmann collision operator and an arbitrary $\phi(\xi)$, we get: 
\begin{multline*}
  \int_{{\rm I\!R}^n} Q(f,f)(x,\xi,t)\phi(\xi) \ d\xi =  \frac{1}{4}\int_{{\rm I\!R}^n}\int_{{\rm I\!R}^n}\int_{S^{n-1}}(f^{'} f^{'}_{*} - f f_{*}) \\ \times \big(\phi(\xi) + \phi(\xi_{*}) - \phi(\xi^{'})  - \phi(\xi^{'}_{*})\big)|{\bf n}.(\xi- \xi_{*})| \ d{\bf n} d\xi_* d\xi
\end{multline*}
Set $\phi(\xi)$ equal to either 1, $|\xi|^2$ or $\xi_{i}$ for $1\leq i \leq n$. It follows from the conservation laws that $\phi(\xi) + \phi(\xi_{*}) - \phi(\xi^{'})  - \phi(\xi^{'}_{*})=0$. We conclude that assumptions of Definition \ref{interaction} are satisfied by $Q$ and hence it is a mesoscopic interaction. The notation $Q(f,f)$ is to emphasize that the collision operator is quadratic. Obtaining a similar notation for $g$ as was used for $f$ leads to: 
\begin{align*}
    Q(f, g) = \frac{1}{2}\int_{{\rm I\!R}^n}\int_{S^{n-1}}(f^{'} g^{'}_{*}+g^{'} f^{'}_{*} - g f_{*}-f g_{*})|{\bf n}.(\xi- \xi_{*})| d{\bf n}d\xi_* 
\end{align*}
Note that the conservation laws imply $|(\xi-\xi_{*}).{\bf n}| = |(\xi^{'}-\xi^{'}_{*}).{\bf n}|$. It is common and useful to break this operator into the gain and loss terms, which are respectively defined as: 
\begin{align*}
     Q^+(f, g) &= \frac{1}{2}\int_{{\rm I\!R}^n}\int_{S^{n-1}}(f^{'} g^{'}_{*}+g^{'} f^{'}_{*})|{\bf n}.(\xi^{'}- \xi^{'}_{*})| \ d{\bf n}d\xi_* \\
      Q^{-}(f, g) &= \frac{1}{2}\int_{{\rm I\!R}^n}\int_{S^{n-1}}(f_{*}g +g_{*}f )|{\bf n}.(\xi- \xi_{*})| \ d{\bf n}d\xi_* 
\end{align*}
And finally we have: 
\begin{gather*}
      Q(f,g) = Q^{+}(f,g) - Q^{-}(f,g)
\end{gather*}
\newpage
\subsection{Scattering solutions}
By starting with an initial value $f_{0}(x,\xi)$ for the Boltzmann equation, we will create scattering solutions associated to this initial value. It is expected that the initial value has finite mass, energy and momentum as defined in (2.1). To accomplish this goal, we will first define our expectation from a frame of reference that acts as an envelope for the scattering solutions. Next we will show there exist natural observer's and scattering norms associated to any described suitable frame. We will use these norms to create unique scattering solutions to the Boltzmann equation starting from initial values within a specific margin of the frame of reference. At the end of the section we will provide an example of a scattering frame which demonstrates utility of the scattering theory developed in Section 3 in the case of the Boltzmann equation. 

\begin{definition}\label{scatteringframe}
We say a solution to the linear transport equation like $\lambda(x,\xi,t)$ is a {\it scattering frame} if the following two conditions are satisfied: 
\begin{equation}
\begin{aligned}
    \sup_{x,\xi}& \frac{1}{\lambda(x,\xi,0)}\int_{0}^{\infty} Q^{+}(\lambda, \lambda)(x+t\xi,\xi,t)dt \\&+
   \sup_{x,\xi} \frac{1}{\lambda(x,\xi,0)}\int_{0}^{\infty}Q^{-}(\lambda, \lambda)(x+t\xi,\xi,t)dt < \infty
\end{aligned}
\end{equation}
And that for any bounded subset of the spatial variable $D \subset {\rm I\!R}^n$:
\begin{equation}
\begin{aligned}
\lim_{z\rightarrow\infty}&\Big(\sup_{x,\xi}\frac{1}{\lambda(x,\xi,0)}\chi_{D}(x)\int_{z}^{\infty}Q^{+}(\lambda, \lambda)(x+t\xi,\xi,t)dt\\&+  \sup_{x,\xi}\frac{1}{\lambda(x,\xi,0)}\chi_{D}(x)\int_{z}^{\infty}Q^{-}(\lambda, \lambda)(x+t\xi,\xi,t)dt\Big) =0        
\end{aligned}
\end{equation}
We say $\lambda(x,\xi,t)$ is a {\it uniform scattering frame}, if we replace (4.2) above with the stronger condition below (without any reference to the bounded sets):
\begin{equation}
\begin{aligned}
\lim_{z\rightarrow\infty}&\Big(\sup_{x,\xi}\frac{1}{\lambda(x,\xi,0)}\int_{z}^{\infty}Q^{+}(\lambda, \lambda)(x+t\xi,\xi,t)dt\\&+  \sup_{x,\xi}\frac{1}{\lambda(x,\xi,0)}\int_{z}^{\infty}Q^{-}(\lambda, \lambda)(x+t\xi,\xi,t)dt\Big) =0        
\end{aligned}
\end{equation}
\end{definition}
\begin{remark}
 In order to interpret the definition above, consider the evolution of particles associated to the initial value $f_0(x,\xi)\leq\lambda(x,\xi,0)$ that are not interacting with one another, therefore particles move along their own trajectories without change. This evolution is described by $f_{0}(x-t\xi,\xi) \leq \lambda(x,\xi,t)$. In this case, any two particles can meet at most once, only under the conditions that their trajectories intersect and timing is ideal. Although these particles are moving with inertia and effects of the collisions are not being implemented in their motion, still one can measure the total amount of possible interaction and that is represented in (4.1). One could argue that in the non linear case, although particles are subject to dispersion, it is still possible for any two particles to meet more than once. This leads to the expectation that, starting from an identical initial value, the non linear case is subject to more collisions compared to the nonexistent collisions of the linear equation. Recall Definition \ref{scattering solution}, in which one expects integrability of interactions and proper decay along the characteristics, then prior to that one should expect a similar condition for the linear case $\lambda$. Therefore the linear equation can be used as a frame of reference for the non linear one. 
\end{remark}
\begin{definition}\label{isoradius}
For a scattering frame $\lambda(x,\xi,t)$ define the {\it isotropic radius} $\alpha$ as below:
\begin{multline*}
    0<\frac{1}{2\alpha}=\sup_{x,\xi} \frac{1}{\lambda(x,\xi,0)}\int_{0}^{\infty} Q^{+}(\lambda, \lambda)(x+t\xi,\xi,t)dt \\+
   \sup_{x,\xi} \frac{1}{\lambda(x,\xi,0)}\int_{0}^{\infty}Q^{-}(\lambda, \lambda)(x+t\xi,\xi,t)dt < \infty
\end{multline*}
\begin{remark} 
The $\alpha$ defined above quantifies the margin for a scattering frame. Recall the prior remark and the conclusion of Section 3. Use of the term isotropic has a strong root in euclidean geometry where it represents uniform scaling, yet it also appears in fluid mechanics, kinetic theory of gases and quantum mechanics. This geometric interpretation is key to understanding isotropy within these different fields. A scattering solution starting within the isotropic radius of a frame will remain within the same frame along its evolution. In the next theorem, we will show how existence of a scattering frame will lead to existence of global in time scattering solutions. 
\end{remark}
\end{definition}
\indent Recall definitions of the observer's norm, scattering norm and scattering solution from Section 3. In the next theorem and under the assumption of having a scattering frame, we will show existence of scattering solutions to the Boltzmann equation following a similar strategy as described at the final discussion of Section 3. At the end of this section we will provide an example of a distribution that satisfies properties of a scattering frame. Any scattering frame like $\lambda(x, \xi, t)$ has a natural associated observer's and scattering norms as defined below:  
\begin{equation}
 \begin{aligned}
    \|f(x,\xi,t)\|_{O}&= \sup_{x,\xi} \frac{1}{\lambda(x,\xi,0)} f(x,\xi,t)  \\
    \|f(x,\xi,t)\|_{S}&=\sup_{x,\xi,t}\frac{1}{\lambda(x,\xi,t)}f(x,\xi,t) \\
\end{aligned}
\end{equation}
\indent The observer's norms is useful for the comparison of solutions to a frame of reference at a specific slice of time and the scattering norm is useful for comparison of solutions in a faraway future or a distant past. Looking back at Definition \ref{sospaces}, the aforementioned norms have their corresponding scattering space:
\begin{align*}
     f \in S(C)\Longleftrightarrow\|f(x,\xi, t)\|_{S} \leq C
\end{align*}

\begin{theorem}\label{bscatteringthm}
Assume $\lambda(x,\xi,t)$ is a scattering frame with an isotropic radius $\alpha$. For any positive $N < \alpha$ and initial value $f_{0}(x-t\xi,\xi)\in S(\frac{1}{2}N)$, there exists a unique scattering solution to the Boltzmann equation like $f \in S(N)$:
\begin{align*}
    \partial_t f(x,\xi,t) + \xi.\nabla_{x}f(x,\xi,t) &= Q(f, f)(x,\xi,t) \\
     f(x,\xi, 0)&=f_{0}(x,\xi)
\end{align*}
 This solution scatters to a unique solution of the linear transport equation $f_{\infty}(x, \xi, t)\in S(N)$ within any compact set of the spatial variable like $D \subset {\rm I\!R}^n$:
\begin{align*}
   \lim_{t \rightarrow \infty}\|\chi_D(x)\big(f(x+t\xi,\xi,t) - f_{\infty}(x,\xi, 0)\big)\|_{O}=0
 \end{align*}
Conversely, for any solution of the linear transport equation like $f_\infty \in S(\frac{1}{2}N)$, there exists a family of scattering solutions to the Boltzmann equation like $u_{z} \in S(N)$ for $z \in [0, \infty)$, scattering respectively to the unique solutions of linear transport equation $f^{z}_{\infty}(x,\xi,t) \in S(N)$:
\begin{align*}
\lim_{t\rightarrow\infty}\|\chi_{D}(x)(u_z(x+t\xi,\xi,t) - f_{\infty}^{z}(x,\xi,0))\|_{O}=0 
\end{align*}
The sequence $f_{\infty}^z$ converges to $f_{\infty}$ in the sense defined below, where $D \subset {\rm I\!R}^n$ is any compact subset of the spatial variable: 
\begin{align*}
\lim_{z\rightarrow\infty}\|\chi_{D}(x-t\xi)(f_{\infty}(x,\xi,t)-f^{z}_{\infty}(x,\xi,t))\|_{S}=0
\end{align*}
\end{theorem}
\begin{proof}
\noindent 
The scattering frame $\lambda(x, \xi, t)$ has its associated observer's and scattering norms as defined in (4.4). With the aid of these norms we estimate the impacts of the gain and loss terms of the collision operator along the characteristics of the Boltzmann equation. Start with the following observation for the gain term and arbitrary $f,g \in S(N)$: 
\begin{multline*}
    \frac{1}{2}\sup_{x,\xi} \frac{1}{\lambda(x,\xi,0)}\int_{0}^{\infty}   Q^{+}(f,g)(x+t\xi,\xi,t)dt \\ =  \frac{1}{2}\sup_{x,\xi} \frac{1}{\lambda(x,\xi,0)}\int_{0}^{\infty}  \int_{S^{n-1}}\int_{{\rm I\!R}^n}  \Big(f(x + t\xi, \xi^{'},t)g(x + t\xi, \xi_{*}^{'},t)\\ +g(x + t\xi, \xi^{'},t)f(x + t\xi, \xi_{*}^{'},t)\Big)|(\xi^{'} - \xi^{'}_{*}).{\bf n}| \ d\xi_{*} d{\bf n}dt
\end{multline*}
Therefore: 
\begin{multline*}
      \frac{1}{2}\sup_{x,\xi} \frac{1}{\lambda(x,\xi,0)}\int_{0}^{\infty}   Q^{+}(f,g)(x+t\xi,\xi,t)dt \\ \leq
    \|f\|_{S}\|g\|_{S} \sup_{x,\xi} \frac{1}{\lambda(x,\xi,0)}\int_{0}^{\infty}  \int_{S^{N-1}}\int_{{\rm I\!R}^n} \lambda(x + t\xi, \xi^{'},t) \\ \times \lambda(x + t\xi, \xi_{*}^{'},t)|(\xi^{'} - \xi^{'}_{*}).{\bf n}| \ d\xi_{*} d{\bf n}dt
\end{multline*}
\noindent Continue by estimating the loss term similarly: 
\begin{multline*}
    \frac{1}{2}\sup_{x,\xi} \frac{1}{\lambda(x,\xi,0)}\int_{0}^{\infty} Q^{-}(f,g)(x+t\xi,\xi,t)dt \\ \leq  
    \|f\|_{S}\|g\|_{S} \sup_{x,\xi} \frac{1}{\lambda(x,\xi,0)}\int_{0}^{\infty}  \int_{S^{N-1}}\int_{{\rm I\!R}^n} \lambda(x + s\xi, \xi,t) \\ \times \lambda(x + t\xi, \xi_{*},t)|(\xi - \xi_{*}).{\bf n}|  \ d\xi_{*} d{\bf n}dt
\end{multline*}
After combining the two estimates above for gain and loss terms, we will conclude the bound  below in terms of the isotropic radius(Definition \ref{isoradius}) of scattering frame $\lambda$: 
\begin{align}
    \|\int_{0}^{\infty}|Q(f,g)(x+t\xi,\xi,t)|dt\|_{O} \leq \frac{1}{2\alpha}\|f\|_{S}\|g\|_{S}
\end{align}
Now consider the solution map defined below for $f\in S(N)$:
\begin{gather*}
 \Phi(f)(x,\xi,t) = f_{0}(x-t\xi, \xi) + \int_{0}^{t} Q(f,f)(x-t\xi+s\xi,\xi,s) ds  \end{gather*}
We will estimate the scattering norm of $\Phi(f)$ using (4.5): 
\begin{multline*}
    \|\Phi(f)(x,\xi, t)\|_{S} =  
    \sup_{t}\|f_{0}(x, \xi) + \int_{0}^{t} Q(f, f)(x + \xi s, \xi, s) ds\|_{O} \\ \leq
    \|f_{0}(x,\xi)\|_{O}  +  \|\int_{0}^{\infty}|Q(f,f)(x+s\xi,\xi,s)|ds\|_{O}  \leq \|f_{0}\|_{O} +\frac{1}{2\alpha}\|f\|_{S}^{2} < N
\end{multline*}
The last inequality is true because $\|f_{0}(x, \xi)\|_{O}=\|f_{0}(x-t\xi,\xi)\|_{S}\leq \frac{1}{2}N$  and $f \in S(N)$, while $N < \alpha$. This computation  implies that the iterations of solution map remain within the same scattering space:  
\begin{align*}
   f \in S(N) \Longrightarrow \Phi(f) \in S(N)
\end{align*}
\noindent Furthermore since $Q(f,f)$ is quadratic, for $u_1, u_2 \in S(N)$ we have:
\begin{multline*}
    \|\Phi(u_{1})(x, \xi, t)-\Phi(u_{2})(x, \xi, t)\|_{S} \\= \|\int_{0}^{t}Q(u_{1},u_{1})(x-t\xi+z\xi,\xi,z)-Q(u_{2},u_{2})(x-t\xi+z\xi,\xi,z)dz\|_{S}\\ =
    \sup_{t}\|\int_{0}^{t}Q(u_1 - u_2,u_1+u_2)(x+z\xi,\xi, z)dz\|_{O}
\end{multline*}
Consequently: 
\begin{multline*}
     \|\Phi(u_{1})-\Phi(u_{2})\|_{S} \leq \|\int_{0}^{\infty}|Q(u_1 - u_2,u_1+u_2)(x+z\xi,\xi, z)|dz\|_{O} \\ \leq  \frac{1}{2\alpha} \|u_1-u_2\|_{S}\|u_1+u_2\|_{S} \leq \frac{N}{\alpha}\|u_{1}-u_{2}\|_{S}
\end{multline*}
We used estimate (4.5) in the previous computation. It follows that:
\begin{align}
\|\Phi(u_{1})-\Phi(u_{2})\|_{S} \leq \frac{N}{\alpha}\|u_{1}-u_{2}\|_{S}    
\end{align} 
\noindent Previously we showed that the iterations of solution map remains within the same scattering space $S(N)$. Using estimate (4.6) and noting that $N < \alpha$, we conclude $\Phi$  is a contraction on $S(N)$. Therefore, the Banach fixed point theorem implies existence of a unique fixed point for the solution map $\Phi$ like $f(x, \xi, t) \in S(N)$, which is a solution of the Boltzmann equation for the initial value $f_{0}(x,\xi)$. Because scattering frame $\lambda(x, \xi, t)$ satisfies (4.1) and (4.2), we conclude that $f$ satisfies the conditions of Definition \ref{scattering solution}, hence it is a scattering solution. By invoking Theorem \ref{scatteringthm} we deduce that $f$ scatters to a unique linear state like $f_{\infty} \in S(N)$ and complete the proof of the first part of theorem. Equivalently for any compact subset like $D \subset {\rm I\!R}^n$ we have: 
\begin{align*}
   \lim_{t \rightarrow \infty}\|\chi_D(x)\big(f(x+t\xi,\xi,t) - f_{\infty}(x,\xi, 0)\big)\|_{O}=0
 \end{align*}
\begin{figure}[t]
    \centering
\tikzset{every picture/.style={line width=0.75pt}} 

\tikzset{every picture/.style={line width=0.75pt}} 

\begin{tikzpicture}[x=0.75pt,y=0.75pt,yscale=-1,xscale=1]

\draw  [dash pattern={on 0.84pt off 2.51pt}]  (267.5,115.17) -- (267.69,155.44) ;
\draw [shift={(267.69,155.44)}, rotate = 269.73] [color={rgb, 255:red, 0; green, 0; blue, 0 }  ][line width=0.75]      (2.24,-2.24) .. controls (1,-2.24) and (0,-1.23) .. (0,0) .. controls (0,1.23) and (1,2.24) .. (2.24,2.24) ;
\draw [shift={(267.5,115.17)}, rotate = 89.73] [color={rgb, 255:red, 0; green, 0; blue, 0 }  ][line width=0.75]      (2.24,-2.24) .. controls (1,-2.24) and (0,-1.23) .. (0,0) .. controls (0,1.23) and (1,2.24) .. (2.24,2.24) ;
\draw  [dash pattern={on 0.84pt off 2.51pt}]  (225.54,171.5) -- (225.4,233) ;
\draw  [dash pattern={on 0.84pt off 2.51pt}]  (267.69,160.19) -- (268.09,236.19) ;
\draw    (132.92,67) -- (132.5,236) ;
\draw    (132.5,195.75) -- (306.33,150.33) ;
\draw [line width=0.75]    (133,236) -- (330,236) ;
\draw [shift={(333,236)}, rotate = 180] [fill={rgb, 255:red, 0; green, 0; blue, 0 }  ][line width=0.08]  [draw opacity=0] (8.93,-4.29) -- (0,0) -- (8.93,4.29) -- cycle    ;
\draw  [fill={rgb, 255:red, 255; green, 255; blue, 255 }  ,fill opacity=1 ] (129.69,195.75) .. controls (129.69,194.2) and (130.95,192.94) .. (132.5,192.94) .. controls (134.05,192.94) and (135.31,194.2) .. (135.31,195.75) .. controls (135.31,197.3) and (134.05,198.56) .. (132.5,198.56) .. controls (130.95,198.56) and (129.69,197.3) .. (129.69,195.75) -- cycle ;
\draw    (183,138.6) .. controls (209.1,131.2) and (201,176.2) .. (225.54,171) .. controls (249,162.6) and (247,113.8) .. (267.5,111.67) .. controls (280.6,104.2) and (283.67,125) .. (306.33,111.67) ;
\draw    (184.2,88.2) .. controls (214.2,84.2) and (244.77,166.58) .. (267.69,160.19) .. controls (290.6,153.8) and (279.4,121.93) .. (301.4,124.6) ;
\draw  [fill={rgb, 255:red, 255; green, 255; blue, 255 }  ,fill opacity=1 ] (264.88,160.19) .. controls (264.88,158.63) and (266.13,157.38) .. (267.69,157.38) .. controls (269.24,157.38) and (270.5,158.63) .. (270.5,160.19) .. controls (270.5,161.74) and (269.24,163) .. (267.69,163) .. controls (266.13,163) and (264.88,161.74) .. (264.88,160.19) -- cycle ;
\draw  [fill={rgb, 255:red, 255; green, 255; blue, 255 }  ,fill opacity=1 ] (222.73,171) .. controls (222.73,169.45) and (223.99,168.19) .. (225.54,168.19) .. controls (227.1,168.19) and (228.36,169.45) .. (228.36,171) .. controls (228.36,172.55) and (227.1,173.81) .. (225.54,173.81) .. controls (223.99,173.81) and (222.73,172.55) .. (222.73,171) -- cycle ;
\draw  [fill={rgb, 255:red, 255; green, 255; blue, 255 }  ,fill opacity=1 ] (222.88,236) .. controls (222.88,234.45) and (224.13,233.19) .. (225.69,233.19) .. controls (227.24,233.19) and (228.5,234.45) .. (228.5,236) .. controls (228.5,237.55) and (227.24,238.81) .. (225.69,238.81) .. controls (224.13,238.81) and (222.88,237.55) .. (222.88,236) -- cycle ;
\draw  [fill={rgb, 255:red, 255; green, 255; blue, 255 }  ,fill opacity=1 ] (265.27,236.19) .. controls (265.27,234.63) and (266.53,233.38) .. (268.09,233.38) .. controls (269.64,233.38) and (270.9,234.63) .. (270.9,236.19) .. controls (270.9,237.74) and (269.64,239) .. (268.09,239) .. controls (266.53,239) and (265.27,237.74) .. (265.27,236.19) -- cycle ;
\draw    (132.92,111.64) .. controls (154.1,166.19) and (187,217) .. (215,217.17) .. controls (243,217.33) and (247.8,161.8) .. (281,176.2) ;
\draw  [dash pattern={on 4.5pt off 4.5pt}]  (281,176.2) .. controls (296.2,181.4) and (298.6,155) .. (326.2,150.6) ;
\draw  [fill={rgb, 255:red, 255; green, 255; blue, 255 }  ,fill opacity=1 ] (130.1,111.64) .. controls (130.1,110.08) and (131.36,108.82) .. (132.92,108.82) .. controls (134.47,108.82) and (135.73,110.08) .. (135.73,111.64) .. controls (135.73,113.19) and (134.47,114.45) .. (132.92,114.45) .. controls (131.36,114.45) and (130.1,113.19) .. (130.1,111.64) -- cycle ;
\draw  [fill={rgb, 255:red, 255; green, 255; blue, 255 }  ,fill opacity=1 ] (264.69,110.7) .. controls (264.69,109.15) and (265.95,107.89) .. (267.5,107.89) .. controls (269.05,107.89) and (270.31,109.15) .. (270.31,110.7) .. controls (270.31,112.25) and (269.05,113.51) .. (267.5,113.51) .. controls (265.95,113.51) and (264.69,112.25) .. (264.69,110.7) -- cycle ;
\draw  [dash pattern={on 4.5pt off 4.5pt}]  (306.33,111.67) .. controls (325.83,99.67) and (317.67,113) .. (335.67,107) ;
\draw  [dash pattern={on 4.5pt off 4.5pt}]  (301.4,124.6) .. controls (313,126.33) and (317.67,128.56) .. (335.67,122.56) ;
\draw  [dash pattern={on 4.5pt off 4.5pt}]  (132.67,178.75) -- (339.67,125) ;
\draw  [dash pattern={on 4.5pt off 4.5pt}]  (133.33,163.75) -- (340.33,110) ;
\draw  [dash pattern={on 4.5pt off 4.5pt}]  (306.33,150.33) -- (339.5,142) ;
\draw    (246.6,114.2) -- (265.3,133.69) ;
\draw [shift={(266.69,135.13)}, rotate = 226.18] [color={rgb, 255:red, 0; green, 0; blue, 0 }  ][line width=0.75]    (4.37,-1.32) .. controls (2.78,-0.56) and (1.32,-0.12) .. (0,0) .. controls (1.32,0.12) and (2.78,0.56) .. (4.37,1.32)   ;
\draw  [dash pattern={on 4.5pt off 4.5pt}]  (154.87,92.87) .. controls (165.8,97) and (165,103.8) .. (184.2,88.2) ;
\draw  [dash pattern={on 4.5pt off 4.5pt}]  (162.2,142.6) .. controls (171.4,143.8) and (177.8,142.2) .. (183,138.6) ;
\draw  [fill={rgb, 255:red, 255; green, 255; blue, 255 }  ,fill opacity=1 ] (130.04,178.75) .. controls (130.04,177.2) and (131.3,175.94) .. (132.85,175.94) .. controls (134.41,175.94) and (135.67,177.2) .. (135.67,178.75) .. controls (135.67,180.3) and (134.41,181.56) .. (132.85,181.56) .. controls (131.3,181.56) and (130.04,180.3) .. (130.04,178.75) -- cycle ;
\draw  [fill={rgb, 255:red, 255; green, 255; blue, 255 }  ,fill opacity=1 ] (130.1,164.11) .. controls (130.1,162.56) and (131.36,161.3) .. (132.92,161.3) .. controls (134.47,161.3) and (135.73,162.56) .. (135.73,164.11) .. controls (135.73,165.66) and (134.47,166.92) .. (132.92,166.92) .. controls (131.36,166.92) and (130.1,165.66) .. (130.1,164.11) -- cycle ;

\draw (104.3,187.8) node [anchor=north west][inner sep=0.75pt]  [font=\normalsize]  {$f_{\infty }$};
\draw (217.9,240) node [anchor=north west][inner sep=0.75pt]  [font=\normalsize]  {$z_{1}$};
\draw (104.92,100) node [anchor=north west][inner sep=0.75pt]  [font=\normalsize]  {${u_{\infty }}$};
\draw (179.39,113.7) node [anchor=north west][inner sep=0.75pt]  [font=\normalsize]  {${u_{z}}_{1}$};
\draw (197.29,73) node [anchor=north west][inner sep=0.75pt]  [font=\normalsize]  {${u_{z}}_{2}$};
\draw (237.49,103) node [anchor=north west][inner sep=0.75pt]  [font=\normalsize]  {$\epsilon $};
\draw (259.7,240) node [anchor=north west][inner sep=0.75pt]  [font=\normalsize]  {$z_{2}$};
\draw (104.92,164.2) node [anchor=north west][inner sep=0.75pt]  [font=\normalsize]  {$f^{z_{2}}_{\infty }$};
\draw (104.92,142.8) node [anchor=north west][inner sep=0.75pt]  [font=\normalsize]  {$f^{z_{1}}_{\infty }$};

\end{tikzpicture}
\begin{adjustwidth}{65pt}{65pt}
\caption{\small $u_{z_1}$ and $u_{z_2}$ are two scattering solutions of the Boltzmann equation, that coincide with $f_{\infty}$ respectively at times $z_1$ and $z_2$. Consider that existence of $u_{\infty}$ is not guaranteed in general.}
\end{adjustwidth}
\end{figure}
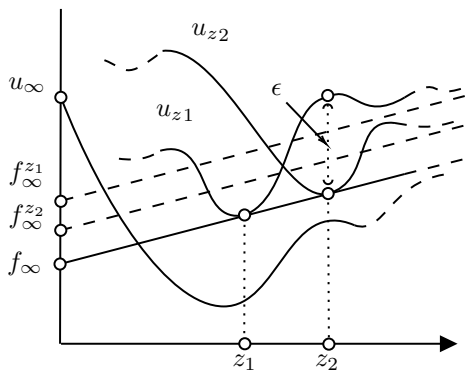

\indent For the converse, assume $f_{\infty}(x,\xi,t) \in S(\frac{1}{2}N)$ is an arbitrary solution to the linear transport equation. Define a family of solution maps for $0 \leq z$ as:
\begin{align*}
    \Phi_{z}(g)(x,\xi,t) =  f_{\infty}(x-(t-z)\xi, \xi,z) +  \int_{z}^{t} Q(g,g)(x-t\xi+s\xi,\xi,s) ds
\end{align*}
Since $f_{\infty} \in S(\frac{1}{2}N)$, using the first part of the proof we conclude that iterations of these  maps converge to a unique scattering solutions of the Boltzmann equation like $u_{z} \in S(N)$. Equivalently we have: 
\begin{align*}
    \lim_{n\rightarrow\infty}\Phi^{n}_{z}&=u_{z}(x,\xi,t) \\ 
    u_{z}(x,\xi,z)&=f_{\infty}(x,\xi,z)
\end{align*}
Theorem \ref{scatteringthm} implies that $u_{z}$ scatters to a unique solution of the linear transport equation like $f_{\infty}^{z} \in S(N)$ within any compact set $D$:
\begin{align}
    \lim_{t\rightarrow\infty}\|\chi_{D}(x)(f^{z}_{\infty}(x,\xi,0) - u_{z}(x+t\xi,\xi,t))\|_{O}=0
\end{align}
We will show that the sequence $\chi_{D}(x-t\xi)f_{\infty}^{z}$ is Cauchy in the scattering space $S(N)$ converging to $f_{\infty}$ with the aid of the forthcoming approximations. Consider Figure 5 and let $\epsilon$ be defined as: 
\begin{align*}
  \epsilon = \| \chi_{D}(x)(u_{z_1}(x+z_2\xi,\xi,z_2) - f_{\infty}(x+z_2\xi,\xi,z_2))\|_{O}  
\end{align*}
For $0\leq z_1 \leq z_2$ we get: 
\begin{multline*}
   \epsilon=\| \chi_{D}(x)\Big(u_{z_1}(x+z_1\xi,\xi,z_1)+ \int_{z_1}^{z_2}Q(u_{z_1},u_{z_1})(x+z\xi,\xi,z)dz\\-f_{\infty}(x+z_1\xi,\xi,z_1)\Big)\|_{O} = \| \chi_{D}(x)\int_{z_1}^{z_2}Q(u_{z_1},u_{z_1})(x+z\xi,\xi,z)dz\|_{O}
\end{multline*}
Consequently: 
\begin{align}
    \epsilon = \| \chi_{D}(x)\int_{z_1}^{z_2}|Q(u_{z_1},u_{z_1})(x+z\xi,\xi,z)| dz\|_{O}
\end{align}

\noindent In addition for $z_1\leq z_2 \leq t$ we have:
\begin{multline*}
    \|\chi_{D}(x)(u_{z_{1}}(x+t\xi,\xi,t) - u_{z_{2}}(x+t\xi,\xi,t))\|_{O}\\ =  \|\chi_{D}(x)\Big(u_{z_1}(x+z_2\xi,\xi,z_2)+\int_{z_2}^{t}Q(u_{z_1}, u_{z_1})(x+z\xi,\xi,z)\\- f_{\infty}(x+z_2\xi,\xi,z_2)-\int_{z_2}^{t}Q(u_{z_2}, u_{z_2})(x+z\xi,\xi,z)dz\Big)\|_{O} 
\end{multline*}
Therefore we get: 
\begin{equation}
\begin{aligned}
    \|\chi_{D}(x)&(u_{z_{1}}(x+t\xi,\xi,t) - u_{z_{2}}(x+t\xi,\xi,t))\|_{O} \\ &\leq \epsilon +
    \|\chi_{D}(x)\int_{z_2}^{t}Q(u_{z_1} - u_{z_1}, u_{z_2} + u_{z_2})(x+z\xi,\xi,z)dz\|_{O}
\end{aligned}
\end{equation}
Now continue with the following computation for $0\leq z_1\leq z_2$, where we used the triangle inequality: 
\begin{multline*}
 \|\chi_{D}(x)(f_{\infty}^{z_1}(x,\xi,0) - f_{\infty}^{z_2}(x,\xi,0))\|_{O} \leq
   \|\chi_{D}(x)(f^{z_1}_{\infty}(x,\xi,0) - u_{z_1}(x+t\xi,\xi,t))\|_{O}\\ + \|\chi_{D}(x)(f^{z_2}_{\infty}(x,\xi,0) - u_{z_2}(x+t\xi,\xi,t))\|_{O} \\ + \|\chi_{D}(x)(u_{z_{1}}(x+t\xi,\xi,t) - u_{z_{2}}(x+t\xi,\xi,t))\|_{O}
\end{multline*}
The left hand side of the previous inequality is independent of $t$. We will compute the limit of the right hand side with respect to $t$ and continue with estimates (4.7) and (4.9), it follows that: 
\begin{multline*}
    \|\chi_{D}(x)(f_{\infty}^{z_1}(x,\xi,0) - f_{\infty}^{z_2}(x,\xi,0))\|_{O}\leq \\  \lim_{t\rightarrow\infty} \|\chi_{D}(x)(u_{z_{1}}(x+t\xi,\xi,t) - u_{z_{2}}(x+t\xi,\xi,t))\|_{O}\leq \\ 
   \epsilon + \|(\chi_{D}(x)\int_{z_2}^{\infty}|Q(u_{z_1} - u_{z_1}, u_{z_2} + u_{z_2})(x+z\xi,\xi,z)|dz\|_{O}
\end{multline*}
Because $u_{z_1}$ and $u_{z_2}$ are scattering solutions in space $S(N)$, we deduce from the inequality above and (4.8) that $\|\chi_{D}(x)(f_{\infty}^{z_1}(x,\xi,0) - f_{\infty}^{z_2}(x,\xi,0))\|_{O}$ goes to zero as $z_1$ and $z_2$ approach the infinity. On the other hand because the family of functions $f_{\infty}^{z}$ are solutions to the linear transport equation we know: 
\begin{multline*}
    \|\chi_{D}(x-t\xi)(f_{\infty}^{z_1}(x,\xi,t) - f_{\infty}^{z_2}(x,\xi,t))\|_{S}= \|\chi_{D}(x)(f_{\infty}^{z_1}(x,\xi,0) - f_{\infty}^{z_2}(x,\xi,0))\|_{O}
\end{multline*}
This implies that $\|\chi_{D}(x-t\xi)f_{\infty}^{z_1}(x,\xi,t) - \chi_{D}(x-t\xi)f_{\infty}^{z_2}(x,\xi,t)\|_{S}$ converges to zero as $z_1$ and $z_2$ go to infinity, therefore $\chi_{D}(x-t\xi)f_{\infty}^{z}$ is a Cauchy sequence converging to another linear state. To complete the proof of we will show that this linear state is in fact $f_{\infty}$. For any compact $D \subset {\rm I\!R}^n$ we have: 
\begin{multline*}
      \lim_{t\rightarrow\infty}\|\chi_{D}(x)(f^{z}_{\infty}(x,\xi,0) - u_{z}(x+t\xi,\xi,t))\|_{O}\\=  \|\chi_{D}(x)\big(f^{z}_{\infty}(x,\xi,0) - u_{z}(x+z\xi,\xi,z)- \int_{z}^{\infty}Q(u_z,u_z)(x+r\xi,\xi,r)dr \big)\|_{O}=0
\end{multline*}
Consequently: 
\begin{align*}
    \chi_{D}(x)f^{z}_{\infty}(x,\xi,0) - \chi_{D}(x)f_{\infty}(x,\xi,0) = \chi_{D}(x)\int_{z}^{\infty}Q(u_z,u_z)(x+r\xi,\xi,r)dr
\end{align*}
Because $u_z$ is a scattering solution, from (3.1) we conclude that the right hand side of the expression above goes to zero as $z$ goes to infinity. This shows that the limit of the previously described Cauchy sequence can not be anything beside $\chi_{D}(x-t\xi)f_\infty(x,\xi,t)$ and complete the proof:
\begin{align*}
\lim_{z\rightarrow\infty}\|\chi_{D}(x-t\xi)(f_{\infty}(x,\xi,t)-f^{z}_{\infty}(x,\xi,t))\|_{S}=0
\end{align*}
\end{proof}
\indent Under the assumption of having a scattering frame, the previous theorem guarantees existence of a class of scattering solutions which exhibit asymptotic stability and completeness. At the end of this section, by providing an example of a scattering frame, this theorem will gain utility. Prior to that, we will investigate the impact of replacing the assumption on existence of a scattering frame in Theorem \ref{bscatteringthm} with a uniform scattering frame.

\begin{corollary}\label{ubscatteringthm} 
Assume $\lambda(x,\xi,t)$ is a uniform scattering frame with isotropic radius $\alpha$. For any positive $N <\alpha$ and initial value $f_{0}(x-t\xi,\xi)\in S(\frac{1}{2}N)$, there exists a unique uniform scattering solution to the Boltzmann equation like $f \in S(N)$ such that scatters uniformly to a unique solution of the linear transport equation like $f_{\infty} \in S(N)$:
\begin{align*}
   \lim_{t \rightarrow \infty}\|f(x+t\xi,\xi,t) - f_{\infty}(x,\xi, 0)\big\|_{O}=0
 \end{align*}
Conversely, for any solution of the linear transport equation like $f_\infty \in S(\frac{1}{2}N)$, there exists a unique uniform scattering solution to the Boltzmann equation like $u_{\infty} \in S(N)$, such that $u_{\infty}$ scatters uniformly to $f_{\infty}$:
\begin{align*}
    \lim_{t\rightarrow\infty}\|u_{\infty}(x+t\xi,\xi,t) - f_{\infty}(x,\xi,0)\|_{O}=0
\end{align*}

\begin{proof}
We will replace the assumption on existence of a scattering frame in the proof of the previous theorem with a uniform scattering frame. Replicating the first part of the proof with the aforementioned stronger assumption, and with the aid of Corollary \ref{uscatteringc} we conclude the existence of uniform scattering solutions. For the converse, assume that the family of solutions $u_z$ is as defined in the proof of Theorem \ref{bscatteringthm}. We will show that $u_{z}$ is a Cauchy sequence in scattering space $S(N)$ converging to a unique uniform scattering solution of the Boltzmann equation $u_{\infty} \in S(N)$.
Start with the computation below: 
\begin{multline*}
     \|u_{z_1} - u_{z_2}\|_{S} = \sup_{t}\| \Big(u_{z_1}(x+z_2\xi,\xi,z_2)+\int_{z_2}^{t} Q(u_{z_1},u_{z_1})(x+z\xi,\xi,z) \ dz\\ - f_{\infty}(x+z_2 \xi, \xi, z_2)-\int_{z_2}^{t} Q(u_{z_2},u_{z_2})(x+z\xi,\xi,z) \ dz\Big)\|_{O}
\end{multline*}
Consider figure 5 and let $\epsilon$ be as below: 
\begin{align*}
    \epsilon=\|(u_{z_1}(x+z_2\xi,\xi,z_2) - f_{\infty}(x+z_2 \xi, \xi, z_2)\|_{O}
\end{align*}
Therefore: 
\begin{multline*}`
     \|u_{z_1} - u_{z_2}\|_{S} \leq \epsilon + \|\int_{0}^{\infty}|Q(u_{z_1}-u_{z_2},u_{z_1}+u_{z_2})(x+z\xi,\xi,z)|dz\|_{O} \\ \leq  \epsilon + \frac{N}{\alpha}\|u_{z_1} - u_{z_2}\|_{S}
\end{multline*}
We conclude that: 
\begin{align}
(1-\frac{N}{\alpha}) \|u_{z_1} -  u_{z_2}\|_{S} \leq \epsilon
\end{align}
Now continue with a computation for $\epsilon$:
\begin{multline*}
  \epsilon =  \|u_{z_1}(x+z_1\xi,\xi,z_1)+ \int_{z_1}^{z_2}Q(u_{z_1},u_{z_1})(x+z\xi,\xi,z) \ dz-f_{\infty}(x+z_1\xi,\xi,z_1)\|_{O}\\ = \| \int_{z_1}^{z_2}Q(u_{z_1},u_{z_1})(x+z\xi,\xi,z) \ dz\|_{O}
\end{multline*}
Since $u_{z}$s are uniform scattering solutions, it is immediate that as $z_{1}$ and $z_{2}$ go to infinity $\epsilon$ goes to zero. From the previous statement and estimate (4.10) we conclude that $u_{z}$ is a Cauchy sequence in $S(N)$ converging to another uniform scattering solution of the Boltzmann equation like $u_{\infty}$. Corollary \ref{uscatteringc} states that this solution scatters uniformly to a unique solution of the linear transport equation.
\end{proof}
\end{corollary}
We conclude by examining an example of a prominent set of distributions, often called Maxwellians and demonstrate that they can be used as a scattering frame for the Boltzmann equation. The following lemma will be used for this purpose. Let $\Omega$ and $\Omega_{D}(z)$ be defined as below, where $D \subset {\rm I\!R}^n$ and $\chi_{D}$ is the indicator function for this set as defined before. 
\begin{equation}
    \begin{aligned} &\Omega_D(z) =\sup_{x, \xi} \chi_{D}(x)\int_{z}^{\infty} \int_{{\rm I\!R}^n} e^{-|x + s\xi - s\xi_{*}|^2 - |\xi_{*}|^2} |\xi - \xi_*|d\xi_{*} ds \\ 
    &\Omega =\sup_{x, \xi}\int_{0}^{\infty} \int_{{\rm I\!R}^n} e^{-|x + s\xi - s\xi_{*}|^2 - |\xi_{*}|^2} |\xi - \xi_*|d\xi_{*} ds 
\end{aligned}
\end{equation}
\begin{lemma}\label{lem} Assume $D$ is an arbitrary bounded subset of the spatial variable and $V_{n}$ is  the volume of $n$ dimensional unit sphere. Let $R$ be any number greater or equal to the diameter of set $D \subset B(0, R) \subset {\rm I\!R}^n$. Then $\Omega_D(z)$ satisfies the  bound below for $z>1$ and $\Omega$ is convergent: 
\begin{align*}
    {\bf i.}& \ \Omega_{D}(z) \leq \int_{{\rm I\!R}^n }\int^{\infty}_{z/\log(z) - R} e^{-s^2-|\xi_{*}|^2} \ dsd\xi_{*}+V_{n}\frac{1}{\log^{n}(z)}\int_{-R}^{\infty}e^{-s^2}ds < \infty \\
    {\bf ii.}&\  \Omega < \infty 
\end{align*}

\begin{proof}
Start with the  following observation where ${\bf n}$ is an element of $n-1$ dimensional unit sphere:
\begin{multline*}
    \Omega_{D}(z)=\sup_{x, \xi} \chi_{D}(x) \int_{{\rm I\!R}^n} \int_{z}^{\infty}e^{-|x + s\xi - s\xi_{*}|^2 - |\xi_{*}|^2} |\xi - \xi_*|dsd\xi_{*} \\ =  \sup_{x, \xi} \chi_{D}(x) \int_{{\rm I\!R}^n}\int_{z|\xi-\xi_{*}|}^{\infty} e^{-|x + s{\bf n}|^2 - |\xi_{*}|^2}dsd\xi_{*}
\end{multline*}
Consequently: 
\begin{multline*}
    \Omega_{D}(z) \leq  \sup_{x, \xi} \chi_{D}(x) \int_{{\rm I\!R}^n}\int_{z|\xi-\xi_{*}|}^{\infty} e^{-|x|^2 - s^2 + 2|x|s - |\xi_{*}|^2}dsd\xi_{*} \\= \sup_{x, \xi}\chi_{D}(x)  \int_{{\rm I\!R}^n}\int_{z|\xi-\xi_{*}|}^{\infty} e^{-(|x|-s)^2 - |\xi_{*}|^2}dsd\xi_{*} \\= \sup_{x, \xi} \chi_{D}(x) \int_{{\rm I\!R}^n} \int_{z|\xi-\xi_{*}|-|x|}^{\infty}e^{-s^2 - |\xi_{*}|^2}dsd\xi_{*} 
\end{multline*}
Because $D \subset B(0,R) \subset {\rm I\!R}^n$, it follows from above that: 
\begin{align*}
    \Omega_{D}(z) \leq  \sup_{x, \xi} \int_{{\rm I\!R}^n}\int_{z|\xi-\xi_{*}| - R}^{\infty}  e^{-s^2 - |\xi_{*}|^2}dsd\xi_{*}
\end{align*}
Now assume that $\epsilon$ is an arbitrary positive number, the previous inequality implies: 
\begin{multline*}
     \Omega_{D}(z) \leq   \sup_{x, \xi}\big(\int_{{\rm I\!R}^n - B(\xi,\epsilon)}\int^{\infty}_{z\epsilon - R} e^{-s^2-|\xi_{*}|^2}dsd\xi_{*}+\int_{B(\xi,\epsilon)}\int_{-R}^{\infty}e^{-s^2-|\xi_{*}|^2}d\xi_{*}\big)  
\end{multline*}
We will substitute $\epsilon$ with $\dfrac{1}{\log(z)}$ under the assumption $1<z$. It follows:
\begin{align*}
    \Omega_{D}(z) \leq \int_{{\rm I\!R}^n }\int^{\infty}_{z/\log(z) - R} e^{-s^2-|\xi_{*}|^2}dsd\xi_{*}+V_{n}\frac{1}{\log^{n}(z)}\int_{-R}^{\infty}e^{-s^2}ds
\end{align*}
Continue by replicating a similar computation for $\Omega$:
\begin{multline*}
    \Omega \leq  \sup_{x, \xi} \int_{{\rm I\!R}^n}\int_{0}^{\infty} e^{-|x|^2 - s^2 + 2|x|s - |\xi_{*}|^2}dsd\xi_{*} \\= \sup_{x, \xi}\int_{{\rm I\!R}^n}\int_{0}^{\infty} e^{-(|x|-s)^2 - |\xi_{*}|^2}dsd\xi_{*} \\= \sup_{x, \xi} \int_{{\rm I\!R}^n} \int_{-|x|}^{\infty}e^{-s^2 - |\xi_{*}|^2}dsd\xi_{*} 
\end{multline*}
We deduce the following bound and complete the proof:
\begin{align*}
   \Omega <  \int_{-\infty}^{\infty} \int_{{\rm I\!R}^n} e^{-s^2 - |\xi_{*}|^2}d\xi_{*}  ds < \infty
\end{align*}
\end{proof}
\end{lemma}
\begin{theorem} The distribution $\lambda(x,\xi,t)= e^{-|x -t\xi|^2-|\xi|^2}$ is a scattering frame with isotropic radius $\alpha$ satisfying  the  bound below, where $V_{n-1}$ is the volume of $n-1$ dimensional unit sphere and $\Omega$ is as (4.11).
\begin{align*}
     \dfrac{1}{4V_{n-1}\Omega}\leq \alpha. 
\end{align*}
\end{theorem}
\begin{proof}
We will show that distribution $\lambda$ satisfies conditions of Definition \ref{scatteringframe}. Start with the estimation below for the loss term:
\begin{multline*}
      \sup_{x,\xi}\frac{1}{\lambda(x,\xi,0)}\chi_{D}(x)\int_{z}^{\infty}Q^{-}(\lambda,\lambda) (x+s\xi, \xi, s)ds \\ \leq  V_{n-1}  \sup_{x,\xi}\chi_{D}(x)\int_{z}^{\infty} \int_{{\rm I\!R}^n}  e^{-|x + s\xi - s\xi_{*}|^2 - |\xi_{*}|^2} |\xi - \xi_*| \ d\xi_{*} ds =
    V_{n-1}\Omega_D(z)
\end{multline*}
Continue by estimating the gain term: 
\begin{multline*}     
  \sup_{x,\xi}\frac{1}{\lambda(x,\xi,0)}\chi_{D}(x)\int_{z}^{\infty}Q^{+}(\lambda, \lambda)(x+s\xi, \xi, s)ds \\ \leq  \sup_{x,\xi}\chi_{D}(x)\int_{z}^{\infty}\int_{S^{n - 1}}\int_{{\rm I\!R}^n}  e^{|x|^2 + |\xi|^2} e^{-|x + s\xi - s\xi^{'}_{*}|^2 - |\xi^{'}_{*}|^2} \\ \times e^{-|x + s\xi - s\xi^{'}|^2 - |\xi^{'}|^2}|(\xi - \xi_{*}).{\bf n}| \  d\xi_{*} d{\bf n} ds
\end{multline*}
We will substitute the identities below in the previous inequality. These identities are results of the conservation laws:
\begin{gather*}
    |x + s\xi - s\xi^{'}_{*}|^2 +|x + s\xi - s\xi^{'}|^2 = |x|^2 + |x + s\xi - s\xi_*|^2 \\
    |\xi|^2 + |\xi_{*}|^2 = |\xi^{'}|^2 + |\xi^{'}_{*}|^2
\end{gather*}
We deduce the following bound for the gain term:
\begin{multline*}
 \sup_{x,\xi}\frac{1}{\lambda(x,\xi,0)}\chi_{D}(x)\int_{z}^{\infty}Q^{+}(\lambda, \lambda)(x+s\xi, \xi, s)ds \\ \leq   V_{n-1} \sup_{x, \xi} \chi_{D}(x)\int_{z}^{\infty} \int_{{\rm I\!R}^n} e^{-|x + s\xi - s\xi_{*}|^2 - |\xi_{*}|^2} |\xi - \xi_*| \ d\xi_{*}ds = V_{n-1}\Omega_D(z)
\end{multline*}
\noindent Continue with combining the estimates for gain and loss terms:
\begin{multline*}
    \sup_{x,\xi}\frac{1}{\lambda(x,\xi,0)}\chi_{D}(x)\int_{z}^{\infty}Q^{+}(\lambda, \lambda)(x + \xi s, \xi, s)ds\\+ \sup_{x,\xi}\frac{1}{\lambda(x,\xi,0)}\chi_{D}(x)\int_{z}^{\infty}Q^{-}(\lambda, \lambda)(x + \xi s, \xi, s)ds \leq 2V_{n-1}\Omega_D(z) 
\end{multline*}
Therefore:
\begin{multline*}
    \lim_{z\rightarrow \infty}\big(\sup_{x,\xi}\frac{1}{\lambda(x,\xi,0)}\chi_{D}(x)\int_{z}^{\infty}Q^{+}(\lambda, \lambda)(x + \xi s, \xi, s)ds\\+ \sup_{x,\xi}\frac{1}{\lambda(x,\xi,0)}\chi_{D}(x)\int_{z}^{\infty}Q^{-}(\lambda, \lambda)(x + \xi s, \xi, s)ds\big) \leq \lim_{z\rightarrow \infty}2V_{n-1}\Omega_D(z)  = 0 
\end{multline*}
Similarly we have:
\begin{multline*}
    \sup_{x,\xi}\frac{1}{\lambda(x,\xi,0)}\int_{0}^{\infty}Q^{+}(\lambda, \lambda)(x + \xi s, \xi, s)ds+\\ \sup_{x,\xi}\frac{1}{\lambda(x,\xi,0)}\int_{0}^{\infty}Q^{-}(\lambda, \lambda)(x + \xi s, \xi, s)ds \leq 2V_{n-1}\Omega < \infty
\end{multline*}
We just showed that $\lambda$ satisfies conditions (4.1) and (4.2) and therefore is a scattering frame. Recall Definition \ref{isoradius}, it follows from the previous computation that:
\begin{multline*}
     \frac{1}{2\alpha}= \sup_{x,\xi}\frac{1}{\lambda(x,\xi,0)}\int_{0}^{\infty}Q^{+}(\lambda, \lambda)(x + \xi s, \xi, s)ds+\\ \sup_{x,\xi}\frac{1}{\lambda(x,\xi,0)}\int_{0}^{\infty}Q^{-}(\lambda, \lambda)(x + \xi s, \xi, s)ds \leq 2V_{n-1}\Omega
\end{multline*}
Finally we deduce the bound below for the isotropic radius of $\lambda$:
\begin{align*}
    \frac{1}{4\Omega V_{n-1}} \leq \alpha
\end{align*}
\end{proof}
\printbibliography
\end{document}